\documentclass{amsart}
\usepackage{amsmath,amssymb,amsthm}
\usepackage[margin=1in]{geometry}
\usepackage{graphicx}
\usepackage{tikz}
\usetikzlibrary{arrows,calc,backgrounds}
\usepackage{float,subfig}
\usepackage{color,soul}
\newcommand\scalemath[2]{\scalebox{#1}{\mbox{\ensuremath{\displaystyle #2}}}}

\newtheorem{prop}{Proposition}
\newtheorem*{prop*}{Proposition}
\newtheorem{thm}{Theorem}
\newtheorem{conj}{Conjecture}
\newtheorem{lem}{Lemma}

\title{Higher Codimension Cycles on the Hilbert Scheme of Three Points on the Projective Plane}
\author{Tim Ryan}
\address{Department of Mathematics, University of Michigan, Ann Arbor, MI 48109}
\email{rtimothy@umich.edu}

\author{Alexander Stathis}
\address{Department of Mathematics, University of Georgia, Athens, GA 30602}
\email{stathis.alexanderj@gmail.com}

\begin{document}
\maketitle
\begin{abstract}In this paper, we study the higher codimensional cycle structure of the Hilbert scheme of three points in the projective plane.
In particular, we compute all Chern/Segre classes of all tautological bundles on it and compute the nef (effective) cones of cycles in codimensions 2 and 3 (dimensions 2 and 3).\end{abstract}

\section{Introduction}
Some of the most important invariants of an algebraic variety are its nef and effective cones of divisors. 
The nef cone controls all embeddings of the variety into projective space, while in nice cases, the effective cone controls all birational contractions of the variety. 
These have been computed for many interesting classes of variety.
In recent years, there has been work generalizing the definitions of the effective and nef cones to the setting of higher codimensional cycles \cite{FL}. 
The effective cone generalizes in a unique way, but the nef cone generalizes into four distinct cones of ``positive'' cycles.
These have begun to be studied in various contexts,
and each new example we are able to compute may help build our intuition and shape future results.
Some examples that have already been studied are powers of elliptic curves and Abelian surfaces \cite{DELV}, moduli spaces of curves with marked points \cite{CC}, spaces of complete conics \cite{LH}, projective bundles over curves \cite{Fu}, and blow ups of projective spaces/Grassmannians \cite{CLO,K}.
This paper will add to to this growing body of examples.

One such case is the Hilbert scheme of points on a surface. More particularly, the Hilbert scheme of points on the projective plane is such a case (nef cone - \cite{LQZ}, effective cone - \cite{CHW}, entire stable base locus decomposition - \cite{LZ}).
Given an integer $n\geq 2$, let $\mathbb{P}^{2[n]}$ be the Hilbert of $n$ points on $\mathbb{P}^2$ which parameterizes subschemes of length $n$, or equivalently with Hilbert polynomial $n$, on $\mathbb{P}^2$.
The generic such subscheme is a collection of $n$ distinct points.
$\mathbb{P}^{2[n]}$ is a smooth projective variety of dimension $2n$, as is the Hilbert scheme of $n$ points on any smooth projective surface \cite{F2}.

Since $\mathbb{P}^{2[n]}$ inherits a group action from $\mathbb{P}^{2}$, the results of Bialynicki-Birula \cite{BB} have been used to compute its homology/Chow ring \cite{EL3, ES, ES2} which are the same as all forms of equivalence are the same on these varieties.

The Chow ring has also been given multiple geometric bases \cite{MS} which make these good candidates for us to compute the cones of higher codimensional cycles.
In addition to the effective cone of cycles, we first consider the nef cone of cycles which is defined to be the cone of cycles dual to the effective cone of complementary codimension cycles under the intersection pairing.
Using these bases, we compute the nef cone in codimension 2 and 3 and dually the effective cone in dimensions 2 and 3 on the Hilbert scheme of three points.
\begin{thm}
$\text{Nef}^2\left(\mathbb{P}^{2[3]}\right)$ and $\text{Nef}^3\left(\mathbb{P}^{2[3]}\right)$ are finite polyhedral cones with 6 and 8 extremal rays and their extremal rays are known in the geometric bases defined by Mallavibarrena and Sols, see Section \ref{sec:msbasis}. Conversely, the same statement holds for $\text{Eff}_2\left(\mathbb{P}^{2[3]}\right)$ and $\text{Eff}_3\left(\mathbb{P}^{2[3]}\right)$ with 6 and 7 classes, respectively.
In particular, these effective cones are closed.
\end{thm}
We defer explicitly stating the extremal rays of the cones as we first need to define the bases of the cycle spaces before that makes any sense, see Section \ref{sec:msbasis}.
The proof of this result utilizes the structure of the orbits of the inherited group action. 
In particular, it reveals much finer information about the variety than the nef and effective cones of divisors, which only depend on the divisorial orbits.
The technique of proof is related to and to some extent formalizes the approach of Lozano Huerta \cite{LH}.
In order to look at some of the other positive cones on the Hilbert scheme, we need to approach the Hilbert schemes from a different perspective utilizing vector bundles on the Hilbert scheme.

The vector bundles on Hilbert schemes of points have also been extensively studied.
In particular, given a vector bundle on the surface, one can pull it back to the universal family over the Hilbert scheme and then push it forward to the Hilbert scheme. 
The Segre classes of these \textit{tautological bundles} have enumerative applications and were the subject of Lehn's conjecture \cite{Le}, which was subsequently proven \cite{V, MOP}.
In particular, Lehn's conjecture gave a generating series for the top Segre classes of tautological bundles coming from line bundles.
This has now been generalized to tautological bundles coming from vector bundles \cite{MOP}.
Conversely, the first Segre/Chern class of a tautological bundle is well known. We then compute all Chern (and hence Segre) classes of all tautological bundles on $\mathbb{P}^{2[3]}$.
Recall that $c_i(V)$ is the $i$-th Chern class of a vector bundle $V$.
\begin{thm}
Given a vector bundle $V$ on $\mathbb{P}^2$, $c_i\left(V^{[3]}\right)$, $1\leq i\leq6$, are known in the geometric bases defined by Mallavibarrena and Sols.
\end{thm}
We again defer the precise statement until we have defined the appropriate bases.
We further compute when some of these tautological bundles are 2-very ample. 
This provides an inner bound for another generalization of the nef cone to higher codimensional cycles, the pliant cone. 

The paper is organized as follows. 
Section 2 covers the preliminaries of Grassmannians, tautological bundles, bases of the Chow ring, cones of cycles, and exceptional bundles.
Section 3 computes the tautological classes of line bundles. 
Section 4 gives lemmas necessary for the computations of the nef cones. 
Section 5 works out the cycle structure of the orbits of $\mathbb{P}^{2[3]}$.
Section 6 contains the proofs of the nef and effective cones. 
Finally, Section 7 concerns 2-very ampleness and the pliant cones.
Additionally, Appendix A lists the formulas for the Chern classes of tautological bundles coming from higher rank vector bundles.

The authors would like to thank Izzet Coskun, Rob Lazarsfeld, Jack Huizenga, C\'esar Lozano Huerta, Jay Kopper, Daniel Levine, Dmitrii Pedchenko and John Sheridan for many useful conversations on the work in this paper. 
The first author was partially supported by NSF grant DMS-1547145.
\section{Preliminaries}

\subsection{Basics About Grassmannians}

We recall some basic facts about Grassmannians. Fix nonzero integers $k$ and $n$ such that $0 < k \leq n$, and let $G = G(k,n)$ be the Grassmannian of $k$-dimensional subspaces of an $n$-dimensional vector space $V$. Given a nonincreasing sequence $\lambda = (\lambda_1,\ldots,\lambda_k)$ of nonnegative integers such that $\lambda_i \leq n-k$ for all $1 \leq i \leq k$ and a fixed flag $$F_0 \subset F_1 \subset \cdots \subset F_{n-1} \subset V$$
of subspaces of $V$, one has the \emph{Schubert variety}
\begin{equation}\label{eqn:schubert}\Sigma_\lambda := \left\{W \in G : \dim(W \cap F_{n-k+i -\lambda_i}) \geq i \text{ for all } 1 \leq i \leq k \right\}.\end{equation}
We denote by $\sigma_\lambda := [\Sigma_\lambda]$ the corresponding class in the Chow ring $A^*(G)$ of $\Sigma_\lambda$. 

Let $\mathcal{V} = V \times G$ be the trivial vector bundle of rank $n$. The Grassmannian admits a tautological sequence
$$0 \to \mathcal{T} \to \mathcal{V} \to \mathcal{Q} \to 0$$
where $\mathcal{T}$ is the tautological subbundle of rank $k$ whose fiber over a point $W \in G$ is the subspace $W \subset V$ and $\mathcal{Q}$ is the tautological quotient bundle of rank $n-k$. It is well known that the Chern classes of the bundles are
$$c_i(\mathcal{T}) = \sigma_{1^i} \text{ and } c_i(\mathcal{Q}) = \sigma_i,\footnote{Here $1^k$ denotes the sequence $(1,\ldots ,1,0,\ldots,0)$ and $i$ is the sequence $(i,0,\ldots,0)$, both of the appropriate lengths.}$$ 
and that the Chern classes of either bundle generates the Chow ring $A^*(G)$. Of course, more is known, see \cite{EH,Fult}. 

\subsection{Tautological bundles and Schur classes}
Let $\mathcal{Z} \subseteq \mathbb{P}^{2[3]} \times \mathbb{P}^2$ be the universal family of the Hilbert scheme, and let $p : \mathcal{Z} \to \mathbb{P}^{2[3]}$ and $q: \mathcal{Z} \to \mathbb{P}^2$ be restricted projections. Given a sheaf $\mathcal{F}$ on $\mathbb{P}^2$, we define the \emph{tautological sheaf} 
$$\mathcal{F}^{[3]} := p_*(q^*(\mathcal{F}))$$
on $\mathbb{P}^{2[3]}$. When $\mathcal{F}$ is the line bundle $\mathcal{O}(d)$ on $\mathbb{P}^2$, the tautological sheaf $\mathcal{O}(d)^{[3]}$ is a vector bundle of rank 3. 

The fiber of $\mathcal{O}(d)^{[3]}$ over a point $Z \in \mathbb{P}^{2[3]}$ is $H^0(\mathcal{O}_Z(d))$. When $d \geq 2$, this space is naturally a quotient 
\begin{equation}\label{eqn:subscheme}0 \to H^0(\mathcal{I}_Z(d)) \to H^0(\mathcal{O}(d)) \to H^0(\mathcal{O}_Z(d)) \to 0\end{equation}
of the global sections of the line bundle $\mathcal{O}(d)$ on $\mathbb{P}^2$, and so the tautological bundle $\mathcal{O}(d)^{[3]}$ induces the regular map
$$\varphi_d:\mathbb{P}^{2[3]} \to G(3,h^0(\mathcal{O}(d)))$$
to the Grassmannian of quotients of $H^0(\mathcal{O}(d))$ defined by sending the point $Z$ to the quotient in Equation \ref{eqn:subscheme}. 
When $d= 2$ this is a divisorial contraction;
when $d > 2$, this map is an embedding \cite{ABCH}. 

Maybe more naturally one views this as a map to the dual Grassmannian $G(h^0(\mathcal{O}(d)) -3 ,h^0(\mathcal{O}(d)))$ of subspaces of $H^0(\mathcal{O}(d))$ defined by sending $Z$ to the kernel $H^0(\mathcal{I}_Z(d))$ in Equation \ref{eqn:subscheme}. Correspondingly, we have that $\mathcal{O}(d)^{[3]}$ is the pullback $\varphi_d^*(\mathcal{Q})$ of the tautological quotient bundle $\mathcal{Q}$ on the Grassmannian $G(h^0(\mathcal{O}(d)) -3 ,h^0(\mathcal{O}(d)))$. In particular, we have that the Chern classes of $\mathcal{O}(d)^{[3]}$ are 
$$c_i(\mathcal{O}(d)^{[3]}) = \varphi_d^*(c_i(\mathcal{Q})).$$
Furthermore, if $\lambda$ is a nonincreasing sequence of nonnegative integers defining a Schubert class $\sigma_\lambda$ in the Chow ring of $G(h^0(\mathcal{O}(d)) -3 ,H^0(\mathcal{O}(d)))$, then the \emph{Schur class} $c_\lambda(\mathcal{O}(d)^{[3]})$ is defined to be
$$c_\lambda(\mathcal{O}(d)^{[3]}) := \varphi_d^*(\sigma_\lambda).$$
Since the Chow ring of the Grassmannian is generated by the Schubert classes $\sigma_k = c_i(\mathcal{Q})$, these are naturally polynomials in the Chern classes $c_i(\mathcal{O}(d)^{[3]})$. 

Note Schur classes are defined analogously for any globally defined vector bundle $V$ on any variety $X$ using the map $X \to G\left(h^0(V)-rk(V),h^0(V)\right)$. Since Chern classes extend to all vector bundles, Schur classes do as well and can be defined directly from the Chern classes using Schubert relations.

\subsection{Geometric Description of the Chern Classes} \label{sec:geometricdescriptions}

In light of Equations \ref{eqn:schubert} and \ref{eqn:subscheme}, we can give a geometric description of the Chern classes $c_i(\mathcal{O}(d)^{[3]})$ for $d \geq 2$. Let $N = h^0(\mathcal{O}_{\mathbb{P}^2}(d)) = \frac{(d+2)(d+1)}{2}$ be the dimension of the space of homogeneous polynomials of degree $d$ in three variables. The map $\varphi_d$ defined in Equation \ref{eqn:subscheme} and viewed as a map to the dual Grassmannian $G(N-3, N)$ of subspaces of $H^0(\mathcal{O}(d))$ sends a scheme $Z \in \mathbb{P}^{2[3]}$ to the subspace $H^0(\mathcal{I}_Z(d))$ of degree $d$ polynomials vanishing completely on $Z$.

The Schubert variety $\Sigma_{j}$ for $1 \leq j \leq 3$ on $G(N-3,N)$ is therefore defined as (via Equation \ref{eqn:schubert})
$$\Sigma_{j} = \left\{W \in G(N-3,N) : \dim(W \cap F_{4-j}) \geq 1\right\}$$
where $F_{4-j}$ is a fixed $4-j$ dimensional subspace of homogeneous polynomials of degree $d$ in three variables. Its pullback $\varphi_d^{-1}(\Sigma_j)$ is therefore the locus of schemes $Z$ in $\mathbb{P}^{2[3]}$ upon which some polynomial in $F_{4-j}$ vanishes completely. 

In slightly different terms, fix a $4-i$ dimensional vector space $W$ of curves of degree $d$ in $\mathbb{P}^2$. The locus $U_W$ of schemes $Z \in \mathbb{P}^{2[3]}$ contained in some curve in $W$ has class
\begin{equation}\label{eqn:geometric} [U_W] = c_i(\mathcal{O}(d)^{[3]}).\end{equation}
One way of fixing a $4-i$ dimensional subspace of curves is to consider the space of curves containing $N-4+i$ general fixed points in $\mathbb{P}^2$.

\subsection{The MS basis}

We now recall the Mallavibarrena and Sols (MS) basis for $A^*(\mathbb{P}^{2[3]})$. For a more complete discussion, see \cite{MS,ME}. There are 22 total classes comprising the MS basis. Two of which are the fundamental class and the class of a point. The Hilbert scheme $\mathbb{P}^{2[3]}$ is smooth, so the Chow group of cycles of dimension $k$ are dual to the Chow group of cycles of codimension $k$ under the intersection pairing.

\subsubsection{Divisors and Curves}

There are two divisor and two curve classes. The first divisor, which we call $H$, is the class of the locus of schemes incident to a general fixed line. The second, which we call $F$, is the class of the locus of schemes containing a subscheme of length two collinear with a general fixed point. 

The first curve class $\varphi$ is the class of the locus of schemes containing two general fixed points and incident to a general fixed line. Let $P \in L$ be a fixed general point $P$ in a fixed general line $L$. The second curve class $\psi$ is the class of the locus of schemes containing $P$, meeting $L$ in length two, and containing a third fixed general point.

See Figure \ref{fig:divsandcurves} for heuristic pictures of the loci representing the classes, and Table \ref{tab:divsandcurves} for the intersection pairing. 

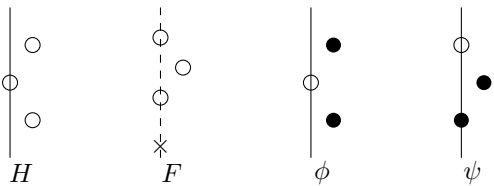
\begin{figure}[h]
    \centering
\begin{tikzpicture}
\draw  (0,-1)--(0,1);
\node[draw=black,circle,scale=.6] at (0,0) {};
\node[draw=black,circle,scale=.6] at (0.3,.5) {};
\node[draw=black,circle,scale=.6] at (0.3,-.5) {};
\node at (.15,-1.2) {$H$};
\draw [dashed]  (2,-1)--(2,1);
\node[draw=black,circle,scale=.6] at (2.3,.2) {};
\node[draw=black,circle,scale=.6] at (2.0,.6) {};
\node[draw=black,circle,scale=.6] at (2.0,-.2) {};
\node at (2,-.85) {$\times$};
\node at (2.15,-1.2) {$F$};
\draw  (4,-1)--(4,1);
\node[draw=black,circle,scale=.6] at (4,0) {};
\node[fill,circle,scale=.6] at (4.3,.5) {};
\node[fill,circle,scale=.6] at (4.3,-.5) {};
\node at (4.15,-1.2) {$\phi$};
\draw  (6,-1)--(6,1);
\node[fill,circle,scale=.6] at (6.3,0) {};
\node[draw=black,circle,scale=.6] at (6,.5) {};
\node[fill,circle,scale=.6] at (6,-.5) {};
\node at (6.15,-1.2) {$\psi$};
\end{tikzpicture}

    \caption{Heuristic pictures of the divisor and curve classes.}
    \label{fig:divsandcurves}
\end{figure}

\begin{table}[h]
    \centering
    \begin{tabular}{c|cc}
        & $\varphi$ & $\psi$ \\ \hline
        $H$ &1 &1\\
        $F$ &2 &1
    \end{tabular}
    \caption{The intersection pairing of the divisors and curves.}
    \label{tab:divsandcurves}
\end{table}

\subsubsection{Fourfolds and Surfaces}
There are five fourfold and  surface classes, i.e. $A^2_\mathbb{Q}\left(\mathbb{P}^{2[3]}\right) \simeq A_{2,\mathbb{Q}}\left(\mathbb{P}^{2[3]}\right) \simeq \mathbb{Q}^5$. Let $L$ and $M$ be general fixed lines, and let $P$ be a fixed point which is in neither line. The first fourfold, which we call $A$, is the class of the locus schemes collinear with $P$. The second, which we call $B$, is the class of the locus of schemes containing a subscheme of length two collinear with $P$ and whose remaining point is incident to $M$. The third, which we call $C$, is the class of the locus of schemes incident to $L$ and $M$ and intersecting their union in length at least two. The fourth, which we call $D$, is the class of the locus of schemes meeting $M$ in length at least two. The fifth, which we call $E$, is the class of the locus of schemes containing $P$.

Let $L$, $M$, and $N$ be lines with $P$ a general point of $L$ and $Q$ a general point of $M$. The first surface class $\alpha$ is the class of the locus of schemes contained in $L$ and containing $P$.  The second surface class $\beta$ is the class of the locus of schemes containing $P$, meeting $L$ in length two, incident to $M$, and contained in $L \cup M$. The third surface class $\gamma$ is the class of the locus of schemes containing $P$, incident to $M$ and $N$, and contained in $M \cup N \cup p$.  The fourth surface class $\delta$ is the class of the locus of schemes containing $P$ and meeting $M$ in length at least two.  The fifth surface class $\epsilon$ is the class of the locus of schemes containing $P$ and $Q$.

See Figure \ref{fig:fourfoldsandsurfaces} for heuristic pictures of the loci representing the classes, and Table \ref{tab:fourfoldsandsurfaces} for the intersection pairing. 

\begin{figure}[h]
    \centering
\begin{tikzpicture}
\draw[dashed]  (0,-1)--(0,1);
\node[draw=black,circle,scale=.6] at (0,0.1) {};
\node[draw=black,circle,scale=.6] at (0,.6) {};
\node[draw=black,circle,scale=.6] at (0,-.4) {};
\node at (0,-.85) {$\times$};
\node at (0,-1.2) {$A$};
\draw [dashed]  (2,-1)--(2,1);
\draw  (2.3,-1)--(2.3,1);
\node[draw=black,circle,scale=.6] at (2.3,.2) {};
\node[draw=black,circle,scale=.6] at (2.0,.6) {};
\node[draw=black,circle,scale=.6] at (2.0,-.2) {};
\node at (2,-.85) {$\times$};
\node at (2.15,-1.2) {$B$};
\draw  (3.85,-1)--(3.85,1);
\draw  (4.15,-1)--(4.15,1);
\node[draw=black,circle,scale=.6] at (3.85,0.5) {};
\node[draw=black,circle,scale=.6] at (4.15,0) {};
\node[draw=black,circle,scale=.6] at (4.45,-.5) {};
\node at (4.15,-1.2) {$C$};
\draw  (6,-1)--(6,1);
\node[draw=black,circle,scale=.6] at (6.3,0) {};
\node[draw=black,circle,scale=.6] at (6,.5) {};
\node[draw=black,circle,scale=.6] at (6,-.5) {};
\node at (6.15,-1.2) {$D$};
\node[fill,circle,scale=.6] at (8.3,0) {};
\node[draw=black,circle,scale=.6] at (8,.5) {};
\node[draw=black,circle,scale=.6] at (8,-.5) {};
\node at (8.15,-1.2) {$E$};
\end{tikzpicture}
\begin{tikzpicture}
\draw  (0,-1)--(0,1);
\node[fill,circle,scale=.6] at (0,0.5) {};
\node[draw=black,circle,scale=.6] at (0,0) {};
\node[draw=black,circle,scale=.6] at (0,-.5) {};
\node at (0,-1.2) {$\alpha$};
\draw  (2,-1)--(2,1);
\draw  (2.3,-1)--(2.3,1);
\node[draw=black,circle,scale=.6] at (2.3,0) {};
\node[fill,circle,scale=.6] at (2.0,.5) {};
\node[draw=black,circle,scale=.6] at (2.0,-.5) {};
\node at (2.15,-1.2) {$\beta$};
\draw  (3.85,-1)--(3.85,1);
\draw  (4.15,-1)--(4.15,1);
\node[draw=black,circle,scale=.6] at (3.85,0.5) {};
\node[draw=black,circle,scale=.6] at (4.15,0) {};
\node[fill,circle,scale=.6] at (4.45,-.5) {};
\node at (4.15,-1.2) {$\gamma$};
\draw  (6,-1)--(6,1);
\node[fill,circle,scale=.6] at (6.3,0) {};
\node[draw=black,circle,scale=.6] at (6,.5) {};
\node[draw=black,circle,scale=.6] at (6,-.5) {};
\node at (6.15,-1.2) {$\delta$};
\node[draw=black,circle,scale=.6] at (8.3,0) {};
\node[fill,circle,scale=.6] at (8,.5) {};
\node[fill,circle,scale=.6] at (8,-.5) {};
\node at (8.15,-1.2) {$\epsilon$};
\end{tikzpicture}
    \caption{Heuristic pictures of the fourfold and surface classes.}
    \label{fig:fourfoldsandsurfaces}
\end{figure}
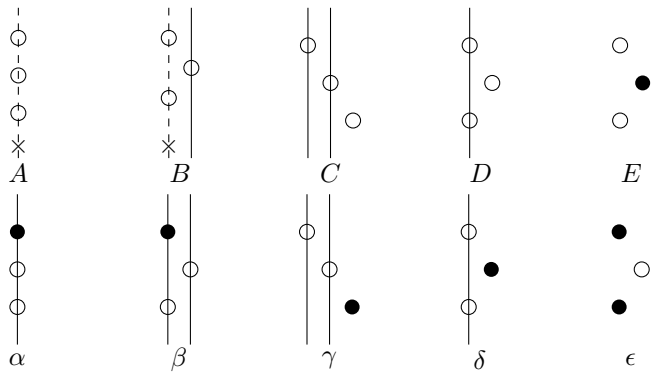

\begin{table}[h]
    \centering
    \begin{tabular}{c|ccccc}
        & $\alpha$ & $\beta$ & $\gamma$ & $\delta$ & $\epsilon$ \\ \hline
        $A$ &0 &0 &1 &0 &0 \\
        $B$ &0 &1 &2 &1 &0 \\
        $C$ &1 &2 &2 &1 &0 \\
        $D$ &0 &1 &1 &0 &0 \\
        $E$ &0 &0 &0 &0 &1 \\
    \end{tabular}
    \caption{The intersection pairing of the fourfolds and surfaces.}
    \label{tab:fourfoldsandsurfaces}
\end{table}

\subsubsection{Threefolds}
There are six threefold classes, i.e. $A^3_\mathbb{Q}\left(\mathbb{P}^{2[3]}\right) \simeq A_{3,\mathbb{Q}}\left(\mathbb{P}^{2[3]}\right) \simeq \mathbb{Q}^6$. Let $L$, $M$, and $N$ be lines with $P$ a general point of $L$ and $Q$ a general point of $M$. The first threefold class $U$ is the class of the locus of schemes containing $P$ and incident to $M$. The second threefold class $V$ is the class of the locus of schemes containing $P$ and containing a subscheme of length two collinear with $Q$. The third threefold class $W$ is the class of the locus of schemes incident to $L$, $M$, and $N$, containing length at least two in each pairwise union of them, and contained in $L \cup M \cup  N$. The fourth threefold class $X$ is the class of the locus of schemes meeting $L$ in length two, incident to $M$, and contained in $L \cup M$. The fifth threefold class $Y$ is the class of the locus of schemes contained in $L$. The sixth threefold class $Z$ is the class of the locus of schemes containing $P$ and meeting $L$ in length at least two.

See Figure \ref{fig:threefoldsandthreefolds} for heuristic pictures of the loci representing the classes, and Table \ref{tab:threefoldsandthreefolds} for the intersection pairing. 

\begin{figure}[h]
    \centering
\begin{tikzpicture}
\draw  (-0,-1)--(-0,1);
\node[draw=black,circle,scale=.6] at (.3,-0.5) {};
\node[draw=black,circle,scale=.6] at (0,0) {};
\node[fill,circle,scale=.6] at (0.3,0.5) {};
\node at (0,-1.2) {$U$};
\draw [dashed]  (2,-1)--(2,1);
\node[fill,circle,scale=.6] at (2.3,.2) {};
\node[draw=black,circle,scale=.6] at (2.0,.6) {};
\node[draw=black,circle,scale=.6] at (2.0,-.2) {};
\node at (2,-.85) {$\times$};
\node at (2.15,-1.2) {$V$};
\draw  (3.85,-1)--(3.85,1);
\draw  (4.15,-1)--(4.15,1);
\draw  (4.45,-1)--(4.45,1);
\node[draw=black,circle,scale=.6] at (3.85,0.5) {};
\node[draw=black,circle,scale=.6] at (4.15,0) {};
\node[draw=black,circle,scale=.6] at (4.45,-.5) {};
\node at (4.15,-1.2) {$W$};
\draw  (6,-1)--(6,1);
\draw  (6.3,-1)--(6.3,1);
\node[draw=black,circle,scale=.6] at (6.3,0) {};
\node[draw=black,circle,scale=.6] at (6,.5) {};
\node[draw=black,circle,scale=.6] at (6,-.5) {};
\node at (6.15,-1.2) {$X$};
\draw  (8,-1)--(8,1);
\node[draw=black,circle,scale=.6] at (8,0) {};
\node[draw=black,circle,scale=.6] at (8,.5) {};
\node[draw=black,circle,scale=.6] at (8,-.5) {};
\node at (8,-1.2) {$Y$};
\draw  (10,-1)--(10,1);
\node[draw=black,circle,scale=.6] at (10.3,0) {};
\node[fill,circle,scale=.6] at (10,.5) {};
\node[draw=black,circle,scale=.6] at (10,-.5) {};
\node at (10.15,-1.2) {$Z$};
\end{tikzpicture}
    \caption{Heuristic pictures of the threefold classes.}
    \label{fig:threefoldsandthreefolds}
\end{figure}
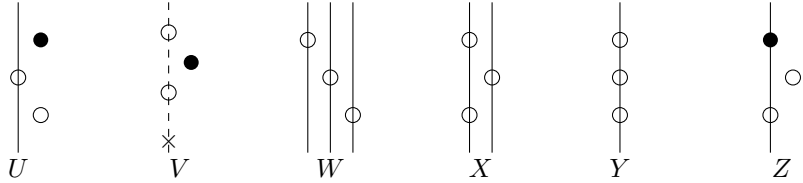

\begin{table}[h]
    \centering
    \begin{tabular}{c|cccccc}
        & $U$ & $V$ & $W$ & $X$ & $Y$ & $Z$ \\ \hline
        $U$ &1 &1 &0 &0 &0 &1\\
        $V$ &1 &1 &0 &0 &0 &0\\
        $W$ &0 &0 &6 &3 &1 &0\\
        $X$ &0 &0 &3 &1 &0 &0\\
        $Y$ &0 &0 &1 &0 &0 &0\\
        $Z$ &1 &0 &0 &0 &0 &1\\
    \end{tabular}
    \caption{The intersection pairing of the threefolds.}
    \label{tab:threefoldsandthreefolds}
\end{table}


\subsection{Cones of cycles}
In this subsection, we introduce the various cones of cycles that we use.
For a more complete discussion, see \cite{FL}.

Let $X$ be a smooth projective variety.
We want to state some generalizations of the effective and nef cones of divisors to higher codimension.
For all of these cones, we start with the vector space of algebraic cycles of dimension (codimension) $k$  up to numerical equivalence, denoted $N_k(X)$ ($N^k(X)$).
First, we generalize the effective cone of divisors.
Inside $N_k(X)$, the pseudoeffective cone $\overline{\text{Eff}}_k(X)$ is the closure of the cone generated by classes of $k$-dimensional subvarieties of X.

We now move to generalizations of the nef cone.
The \textit{nef} cone of cycles is motivated by the intersection pairing property of the nef cone with respect to effective curves.
Inside $N^k(X)$, $\text{Nef}^k(X)$ is defined, since $X$ is smooth, as the dual of $\overline{\text{Eff}}_k(X)$ with respect to the intersection pairing.
Similarly, the \textit{pliant} cone is motivated by analogy to the property that globally generated line bundles are nef.
Inside $N^k(X)$, $\text{Pl}^k(X)$ is defined as the closure of
the cone generated by products of classes, with total codimension k, which are each the pullback of an effective class from a map to a Grassmannian.

\subsection{Exceptional bundles}
In this subsection, we recall a special class of vector bundles and some of their properties.
For a more complete introduction to exceptional bundles, see \cite{DLP}.
An \textit{exceptional} bundle is a vector bundle $V$ such that $\mathrm{ext}^i(V,V) = \delta_{i,0}$.
We say that a slope $\mu$ is exceptional if there exists an exceptional bundle with that slope where the slope of a vector bundle $V$ is given by $\mu(V) = \frac{c_1(V)}{\text{rk}(V)}$.
For each exceptional slope $\alpha = \frac{c_1}{r}$ (where $c_1$ and $r$ are coprime) on $\mathbb{P}^2$ there exists a unique exceptional bundle $E_\alpha$ with $\Delta\left(E_\alpha\right) = \frac{1}{2}(1-\frac{1}{r^2})$ where for a vector bundle $V$ the discriminant $\Delta(V)$ is given by $\Delta(V) = \frac{1}{2}\mu(V)^2-\frac{\text{ch}_2(V)}{\text{rk}(V)}$.
The set of exceptional slopes, $\mathfrak{E}$, is well understood \cite{DLP}.
There exists a bijection $\epsilon:\mathbb{Z}\left[ \frac{1}{2}\right] \to \mathfrak{E}$ defined inductively by $\epsilon(n) = n$ and \[\epsilon \left( \frac{p}{2^{q}}\right) = \epsilon \left( \frac{p-1}{2^{q}}\right) . \epsilon \left( \frac{p+1}{2^{q}}\right)\] where, given two exceptional slopes $\alpha$ and $\beta$, we define \[\alpha . \beta = \frac{\alpha+\beta}{2} + \frac{\Delta_\beta-\Delta_\alpha}{3+\alpha-\beta}.\]
Note, $\epsilon\left( \frac{p}{2^q}\right) < \epsilon\left( \frac{r}{2^{s}}\right)$ iff $\frac{p}{2^q} <  \frac{r}{2^{s}}$.

On $\mathbb{P}^2$, exceptional bundles make up sets called \textit{(strong full) exceptional collections}. A (strong full) exceptional collection on $\mathbb{P}^2$ is a set of the form $\{E_\alpha,E_\beta,E_\gamma\}$ where those are exceptional bundles such that \[\mathrm{ext}^i(E_\beta,E_\alpha)= \mathrm{ext}^i(E_\gamma,E_\alpha) = \mathrm{ext}^i(E_\gamma,E_\beta) =0 \text{ for  all $i$ and } \]
\[\mathrm{ext}^i(E_\alpha,E_\beta)= \mathrm{ext}^i(E_\alpha,E_\gamma) = \mathrm{ext}^i(E_\beta,E_\gamma) =0 \text{ for $i \neq 0$.}\]
We abuse notation and refer to the exceptional collection as $\{\alpha, \beta, \gamma\}$.
Note that $\{\alpha, \beta, \gamma\}$ is an exceptional collection iff $\{\beta -3, \alpha, \beta\}$ is as well.
Every exceptional vector bundle on $\mathbb{P}^2$ can be written uniquely as the middle bundle in an exceptional collection of the form $\{\alpha, \alpha\cdot\beta,\beta \}$.
Every exceptional collection can be gotten from the exceptional collection $\{\mathcal{O},\mathcal{O}(1),\mathcal{O}(2)\}$ via a process called mutation. 
A mutation of exceptional collection $\{\alpha, \beta, \gamma\}$ is where you replace either $\alpha$ and $\beta$ (or $\beta$ and $\gamma$) with a mutation of the pair. 
A mutation of a pair $\alpha$ and $\beta$ is one of the pairs $\beta$, $\delta$ or $\zeta$, $\alpha$ where $E_\delta$ and $E_\zeta$ sit in the following short exact sequences:\[0 \to E_\alpha \to E_\beta^{\chi(E_\alpha,E_\beta)}\to E_\delta \to 0 \text{ or}\] \[0 \to E_\zeta \to E_\alpha^{\chi(E_\alpha,E_\beta)}\to E_\beta \to 0.\] 

Lastly, we want to introduce a property of (exceptional) vector bundles which we use to study the pliant cone later on.
A vector bundle $V$ of rank $r$ on a smooth projective variety $X$ is $(k-1)$\textit{-very ample} if $h^0(V \otimes I_Z) = h^0(V) - rk$ for all subschemes $Z$ of length $k$.

\section{Tautological classes of lines bundles}
In this section, we prove explicit formulas for the Chern (Schur) classes of tautological bundles on $\mathbb{P}^{2[3]}$ coming from line bundles. We first give the classes, then give the intuitive intersection numbers which determine those classes, and then finally give the transversality arguments proving that those intersections are correct.

\begin{thm}
The Chern classes of the tautological bundle $\mathcal{O}(d)^{[3]}$ in the MS basis are as follows:
\[c_1\left(\mathcal{O}(d)^{[3]}\right) = (d-2)H+F,\]
\[c_2\left(\mathcal{O}(d)^{[3]}\right) = A+(d-1)B+\binom{d-1}{2}C+(d-1)D\text{, and}\]
\[c_3\left(\mathcal{O}(d)^{[3]}\right)=  \binom{d}{3}W+2\binom{d}{2}X+dY.\]
Note, all higher Chern classes are zero as $\mathcal{O}(d)^{[3]}$ has rank 3.
\end{thm}

Given the classes of the tautological Chern classes, we can compute the Schur classes using Schubert calculus. 
Recall, the Schur classes are defined from the Chern classes by the relations in Schubert calculus. 
Given that definition and the ring structure \cite{FG}, the following theorem is immediate.
These classes were computed with the aid of Macaulay2.

\begin{thm}
The non-zero Schur classes of the tautological bundle $\mathcal{O}(d)^{[3]}$ in the MS basis are as follows:
\begin{center}
\begin{tabular}{|c|c|}\hline
          Schur class        & Class for $\mathcal{O}(d)^{[3]}$\\ \hline
          $c_{1,1}$        & $2A+(d-2)B+\binom{d-2}{2}C +(3d-5)D+(d^2-2)E$ \\\hline 
          $c_{2,1}$        & $(d^3-\frac{5}{2}d^2-\frac{1}{2}d+3)U+(d^2-2)V+(\frac{1}{3}d^3-2d^2+\frac{11}{3}d-2)W$\\
                           & $+(4d^2-11d+7)X+(8d-9)Y+(\frac{3}{2}d^2-\frac{3}{2}d-1)Z$ \\\hline
          $c_{1,1,1}$      & $(d^3-4d^2+d+6)U+(d^2-4)V+(\frac{1}{6}d^3-\frac{3}{2}d^2+\frac{13}{3}d-4)W$\\
                           & $+(3d^2-13d+14)X+(10d-18)Y+(3d^2-6d+1)Z$ \\\hline
          $c_{3,1}$        & $(2d^2-3d)\alpha+(\frac{3}{2}d^3-\frac{7}{2}d^2+2)\beta+(\frac{1}{2}d^4-2d^3+\frac{5}{2}d^2-d)\gamma+(d^3-2d^2+d)\delta$ \\\hline      
          $c_{2,2}$        & $(4d^2-9d+5)\alpha+(\frac{3}{2}d^3-\frac{11}{2}d^2+6d-2)\beta+(\frac{1}{2}d^4-3d^3+\frac{11}{2}d^2-3d)\gamma$\\
                           & $+(3d^3-7d^2+3d+1)\delta+(\frac{1}{2}d^4-\frac{3}{2}d^2+1)\epsilon$ \\\hline
          $c_{2,1,1}$      & $(10d^2-27d+16)\alpha+(\frac{9}{2}d^3-\frac{35}{2}d^2+21d-7)\beta+(d^4-\frac{13}{2}d^3+\frac{27}{2}d^2-9d)\gamma$\\
                           & $+(4d^3-13d^2+9d+2)\delta+(\frac{1}{2}d^4-3d^2+\frac{3}{2}d+2)\epsilon$ \\\hline
          $c_{1,1,1,1}$    & $(10d^2-36d+32)\alpha+(3d^3-16d^2+27d-14)\beta+(\frac{1}{2}d^4-\frac{9}{2}d^3+13d^2-12d)\gamma$\\
                           & $+(3d^3-13d^2+12d+4)\delta+(\frac{1}{2}d^4-6d^2+\frac{15}{2}d+1)\epsilon$ \\\hline
          $c_{3,2}$        & $(\frac{1}{2}d^5-\frac{3}{2}d^4+3d^2-2d)\phi+(\frac{3}{2}d^4-\frac{3}{2}d^3-3d^2+3d)\psi$ \\\hline
          $c_{3,1,1}$      & $(\frac{1}{2}d^5-\frac{3}{2}d^4-\frac{3}{2}d^3+\frac{17}{2}d^2-7d)\phi+(\frac{3}{2}d^4-\frac{3}{2}d^3-7d^2+9d)\psi$ \\\hline
          $c_{2,2,1}$      & $(d^5-\frac{9}{2}d^4+\frac{45}{2}d^2-31d+12)\phi+(\frac{9}{2}d^4-\frac{15}{2}d^3-18d^2+42d-21)\psi$ \\\hline
          $c_{2,1^3}$      & $(d^5-\frac{9}{2}d^4-\frac{9}{2}d^3+48d^2-73d+30)\phi+(\frac{9}{2}d^4-\frac{15}{2}d^3-36d^2+90d-48)\psi$ \\\hline
          $c_{1^5}$        & $(\frac{1}{2}d^5-3d^4-3d^3+\frac{99}{2}d^2-101d+60)\phi+(3d^4-6d^3-39d^2+126d-96)\psi$ \\\hline
          $c_{3,3}$        & $\frac{1}{6}d^6-\frac{1}{2}d^4+\frac{1}{3}d^2$ \\\hline
          $c_{3,2,1}$      & $\frac{1}{3}d^6-\frac{5}{2}d^4+\frac{3}{2}d^3+\frac{11}{3}d^2-3d$ \\\hline
          $c_{3,1^3}$      & $\frac{1}{6}d^6-2d^4+\frac{3}{2}d^3+\frac{16}{3}d^2-6d$ \\\hline
          $c_{2^3}$        & $\frac{1}{6}d^6-2d^4+\frac{3}{2}d^3+\frac{22}{3}d^2-12d+5$ \\\hline
          $c_{2,2,1,1}$    & $\frac{1}{2}d^6-\frac{15}{2}d^4+9d^3+18d^2-36d+16$ \\\hline
          $c_{2,1^4}$      & $\frac{1}{3}d^6-7d^4+9d^3+\frac{89}{3}d^2-66d+32$ \\\hline
          $c_{1^6}$        & $\frac{1}{6}d^6-5d^4+\frac{15}{2}d^3+\frac{103}{3}d^2-96d+64$ \\\hline
\end{tabular}
\end{center}
\end{thm}

We would be remiss if we did not draw the reader's attention to a pattern that emerges from this chart. The coefficients of each Schur class are polynomials in the class of the original line bundle. The degree of that polynomial at most the codimension of the Schur class. 

\begin{conj}
Let $S$ be a smooth connected surface with $h^1(O_S)=0$ whose Picard group is generated by line bundles $\mathcal{L}_1$, $\cdots$, $\mathcal{L}_m$. Consider a line bundle $\mathcal{L} = \mathcal{L}_1^{a_1} \otimes \cdots \otimes \mathcal{L}_k^{a_m}$. Then the coefficients of any Schur class of $\mathcal{L}^{[n]}$ of codimension $k$ (in any basis for the Chow ring of $S^{[n]}$) are polynomials of degree at most $k$ in the $a_i$.
\end{conj}

In the case where $S$ is the projective plane, we make a stronger conjecture giving the exact class in the MS basis of the Chern classes of tautological bundles.

\begin{conj}
Using the notation of Mallavibarrena and Sols \cite{MS}, 
we have that 
\[c_i\left(\mathcal{O}(d)^{[n]}\right) = \sum \delta_{0,a_1}\binom{d-(n-i)}{f_1+\cdots+f_m}\binom{f_1+\cdots+f_m}{f_1;\cdots;f_m} \sigma_{(a_1^{d_1},\cdots,a_k^{d_k}),(b_1^{e_1},\cdots,b_l^{e_l}),(c_1^{f_1},\cdots,c_k^{f_m})}\]
\end{conj}

\subsection{Intersection Numbers with the MS Basis}
\label{sec:msbasis}
In order to prove the formulas for the Chern classes, it suffices to compute the intersection number of each class with the MS basis classes of complementary codimension. We count the number of points in each set theoretic intersection given general representatives, and then show that the intersection at each such point is transverse.

\begin{prop}
The intersection numbers of $c_1(\mathcal{O}(d)^{[3]})$, $c_2(\mathcal{O}(d)^{[3]})$, and $c_3(\mathcal{O}(d)^{[3]})$ with the classes of complementary codimension in the MS basis are contained in the following tables.
\begin{center}
\begin{tabular}{c|cc}
                             & $\varphi$ & $\psi$ \\ \hline
$c_1\left(\mathcal{O}(d)^{[3]}\right)$   &  $d$  &  $d-1$
\end{tabular}
\end{center}
\begin{center}
\begin{tabular}{c|ccccc}
                              & $\alpha$ & $\beta$ & $\gamma$ & $\delta$ & $\epsilon$ \\ \hline
$c_2\left(\mathcal{O}(d)^{[3]}\right)$    &  $\binom{d-1}{2}$  &  $d(d-1)$  &  $d^2$  &  $\binom{d}{2}$  &  $0$
\end{tabular}
\end{center}
\begin{center}
\begin{tabular}{c|cccccc}
                             & $U$ & $V$ & $W$ & $X$ & $Y$ & $Z$ \\ \hline
$c_3\left(\mathcal{O}(d)^{[3]}\right)$   &  $0$  &  $0$  &  $d^3$  &  $d\binom{d}{2}$  &  $\binom{d}{3}$  &  $0$
\end{tabular}
\end{center}
\end{prop}

The intersection numbers for the first Chern class were computed in \cite{ABCH}, and we omit the calculation here.

\subsubsection{The second Chern class}
Recall that $c_2(\mathcal{O}(d)^{[3]})$ has a geometric representative $\Theta_{2,d}$ defined as follows. Set $N = \frac{(d+2)(d-1)}{2}$, and let $P$ be the two dimensional vector subspace of homogeneous polynomials in $k[\mathbb{P}^2]$ of degree $d$ defined as the polynomials which vanish on $N-2$ general points in $\mathbb{P}^2$. Let $\Theta_{2,d}$ be the locus of schemes $Z \in \mathbb{P}^{2[3]}$ such that some member $p(Z) = 0$ for some $p \in P$. The class $[\Theta_{2,d}] = c_2(\mathcal{O}(d)^{[3]})$. See Section \ref{sec:geometricdescriptions}. 

A general representative of the class $\epsilon$ is given as the locus of schemes containing two fixed general points in $\mathbb{P}^2$. Its intersection with $\Theta_{2,d}$ is therefore empty since there are no polynomials of degree $d$ vanishing on $N$ general points. We thus conclude $\epsilon \cdot c_2(\mathcal{O}(d)^{[3]}) = 0$. 

For the remainder of the intersections, fix two general lines $L_1$ and $L_2$, a general point $P_1$ in $L_1$, and a general point $P_2$. Let $U_\alpha$ be the locus of triples in $\mathbb{P}^{2[3]}$ contained in $L_1$ and containing $P_1$. Let $U_\beta$ be the locus of triples contained in the union $L_1 \cup L_2$ containing $P_1$, meeting $L_1$ in at least length two, and incident to $L_2$. Let $U_\gamma$ be the locus of triples containing $P_2$, incident to each $L_i$ and meeting the union $L_1 \cup L_2$ in length two. Finally, let $U_\delta$ be the locus of triples containing $P_2$ and meeting $L_2$ in length two. Of course, $\alpha = [U_\alpha]$, $\beta = [U_\beta]$, $\gamma = [U_\gamma]$, and $\delta = [U_\delta]$ (see Section \ref{sec:msbasis}). Furthermore, let $C_1$ be the unique curve of degree $d$ contained in $P$ and containing $P_1$, and let $C_2$ be the unique curve of degree $d$ contained in $P$ and containing $P_2$. Note that the fixed curves $C_1$ and $C_2$ meet each of the fixed general lines $L_1$ and $L_2$ transversally. 

We first show that the intersections of $\Theta_{2,d}$ and the loci $U_\alpha$, $U_\beta$, $U_\gamma$, and $U_\delta$ occur at the correct number of points, each of which corresponds to a reduced subscheme of $\mathbb{P}^2$. Transversality of these intersections will be addressed below.

The intersection of $\Theta_{2,d}$ with $U_\alpha$ consists of only schemes contained in the curve $C_1$. 
The additional requirement that the schemes be contained in $L_1$ determines the intersection to be triples chosen from $C_d \cap L_1$. There are ${d-1 \choose 2}$ such schemes, due to the required inclusion of $P_1$ in each such scheme, all consisting of three distinct points.

The intersection of $\Theta_{2,d}$ with $U_\beta$ also consists of only schemes contained in the curve $C_1$ and containing $P_1$. $C_1$ meets each line $L_i$ in $d$ points, one such intersection being $P_1$ itself. Thus the intersection of $\Theta_{2,d}$ with $U_\beta$ consists of $(d-1) \cdot d$ schemes consisting of reduced subschemes of $\mathbb{P}^2$ consisting of two points on $L_1$, one of which is $P_1$, and a point on $L_2$. 

Having now sensed the theme, we calculate the intersections $\Theta_{2,d} \cap U_\gamma$ and $\Theta_{2,d} \cap U_\delta$ relying heavily on the requirement that each scheme in either intersection must be contained in the curve $C_2$. The intersection of $C_2$ with each $L_i$ is $d$ points each, and a choice of one from each such collection determines a reduced subscheme of $\mathbb{P}^2$ in the intersection $\Theta_{2,d} \cap U_\gamma$ for a total of $d^2$ such schemes. A choice of any two from $L_2$ determines a scheme in the intersection $\Theta_{2,d} \cap U_\delta$ for a total of ${d \choose 2}$ reduced subschemes of $\mathbb{P}^2$. 

It remains then to justify the transversality at each point of each intersection. First, let $Z$ be a general reduced subscheme of $\mathbb{P}^2$ of length three -- that is, three distinct points in $\mathbb{P}^2$. Around such a point the Hilbert scheme $\mathbb{P}^{2[3]}$ has charts isomorphic to $\mathbb{A}^2 \times \mathbb{A}^2 \times \mathbb{A}^2$ with coordinates we denote by $x_i,y_i$ for each factor $i = 1,2,3$. Since every point in the intersection of the locus $\Theta_{2,d}$ and any of the loci $U_\alpha, U_\beta, U_\gamma, U_\delta$ is a reduced scheme, we describe the loci explicitly in these charts and compute the transversality directly. We only work out the case of $U_\gamma \cap \Theta_{2,d}$ completely, as the others are analogous. 

Choose as a basis $\{C_1,C_2\}$ for the plane of polynomials $P$ defining the locus $\Theta_{2,d}$. A point $(x_i,y_i)$ is a member of $\Theta_{2,d}$ if there is some choice of $a,b$ such that 
$$aC_1(x_i,y_i) + b C_2(x_i,y_i) = 0$$
for each $i = 1,2,3$ which determines three equations defining $\Theta_{2,d}$:
$$f_{ij} := C_1(x_i,y_i)C_2(x_j,y_j) - C_1(x_j,y_j)C_2(x_i,y_i) = 0$$
for $1 \leq i < j \leq 3$. 

Recall that each point $Z \in U_\gamma \cap \Theta_{2,d}$ is contained in the curve $C_2$, that the curves $C_1$ and $C_2$ are not tangent to either $L_1$ or $L_2$, and that $C_1$ and $C_2$ are distinct independent curves, so that in particular, $C_1$ does not vanish at any point in $Z$. The tangent space at $Z$ is the kernel of the Jacobian
\[\scalemath{0.8}{\left[\begin{array}{cccccc} -\frac{\partial C_2}{\partial x_1} (z_1) \cdot C_1(z_2) & \frac{\partial C_2}{\partial y_1} (z_1) \cdot C_1(z_2) & \frac{\partial C_2}{\partial x_2} (z_2) \cdot C_1(z_1) & \frac{\partial C_2}{\partial y_2} (z_2) \cdot C_1(z_1) &0&0\\
                              -\frac{\partial C_2}{\partial x_1} (z_1) \cdot C_1(z_3) & -\frac{\partial C_2}{\partial y_1} (z_1) \cdot C_1(z_3) & 0 & 0 &\frac{\partial C_2}{\partial x_3} (z_3) \cdot C_1(z_1) & \frac{\partial C_2}{\partial y_3} (z_3) \cdot C_1(z_1) \\ 
                              0&0&-\frac{\partial C_2}{\partial x_2} (z_2) \cdot C_1(z_3) &-\frac{\partial C_2}{\partial y_2} (z_2) \cdot C_1(z_3) & -\frac{\partial C_2}{\partial x_3} (z_3) \cdot C_1(z_2)&\frac{\partial C_2}{\partial y_3} (z_3) \cdot C_1(z_2) \end{array}  \right]}\]
which gives equations
\[\scalemath{.9}{ -\left(\frac{\partial C_2}{\partial x_i} (z_i) \cdot C_1(z_j) \right) x_i - \left(\frac{\partial C_2}{\partial y_i} (z_i) \cdot C_1(z_j) \right) y_i +\left(\frac{\partial C_2}{\partial x_j} (z_j) \cdot C_1(z_i) \right) x_j +\left(\frac{\partial C_2}{\partial y_j} (z_j) \cdot C_1(z_i)\right) y_j = 0} \]
for the tangent space for each $i,j$ with $1 \leq i < j \leq 3$. 

Now every scheme in $U_\gamma$ contains $P_2$, and so we may assume without loss that $P_2 = z_1 \in Z$. Consequently, the equations defining the tangent space to $U_\gamma$ in the first factor of $\mathbb{A}^2$ are $x_1 = y_1 = 0$, so we intersect these with the equations above for $i = 1, j = 2,3$ and achieve
$$\left(\frac{\partial C_2}{\partial x_j} (z_j) \cdot C_1(z_1) \right) x_j +\left(\frac{\partial C_2}{\partial y_j} (z_j) \cdot C_1(z_1)\right) y_j = 0.$$
But since $C_1(z_1) \neq 0$, we see that the tangent space to $\Theta_{2,d}$ in the second and third factors of the chart is simply the tangent space to $C_2$ at $z_2$ and $z_3$, respectively. Since $L_1$ and $L_2$ intersect $C_2$ transversally, the intersection of $\Theta_{2,d}$ with $U_\gamma$ is also transverse. 

\subsubsection{The third Chern class}

Recall that $c_3(\mathcal{O}(d)^{[3]})$ has a geometric representative $\Theta_{3,d}$ described as the locus of triples contained in a curve $C_d$ of degree $d$ (see Section \ref{sec:geometricdescriptions}). 

The intersections of $\Theta_{3,d}$ with any loci representing the classes $U,V$, or $Z$ in the MS basis is zero, since the defining representative loci for each of these classes consist of all schemes containing a fixed general point which is not incident to the curve $C_d$. 

Let $L_1$, $L_2$, and $L_3$ be three general fixed lines in $\mathbb{P}^2$. Let $U_W$ be the locus of triples which are incident to each $L_i$, have length two contained in each pairwise union of lines, and are contained in the union of all three line. 
Let $U_X$ be the locus of triples meeting $L_1$ in length two and incident to $L_2$ and contained in the union $L_1 \cup L_2$, and let $U_Y$ be the locus of triples contained in just $L_1$ (so that $[U_W] =W, [U_X] = X,$ and $[U_Y]=Y$). In particular, each $L_i$ meets $C_d$ transversally in $d$ distinct points. 

The intersection of $\Theta_{3,d}$ with the loci $U_W$ consists of $d^3$ reduced subschemes given as the union of any intersection point of each the three lines $L_i$ with $C_d$. The intersection $\Theta_{3,d} \cap U_X$ consists of ${d \choose 2} \cdot d$ reduced subschemes corresponding to the union of any two points in $L_1 \cap C_d$  and any point of $L_2 \cap C_d$. The intersection $\Theta_{3,d} \cap U_Y$ consists of ${d \choose 3}$ reduced subschemes given as the choice of any three points in $L_1 \cap C_d$. 

Since any point $Z = \{p_1,p_2,p_3\}$ in the intersection of $\Theta_{3,d}$ and $U_W$, $U_X$, or $U_Y$, respectively, is necessarily reduced, there are charts for $\mathbb{P}^{2[3]}$ centered at $Z$ isomorphic to $(\mathbb{A}^2)^3$. Furthermore, $Z$ is necessarily disjoint from the singular locus of $C_d$, so the tangent space to $\Theta_{3,d}$ at $Z$ is the product of the lines tangent to $C_d$ at each of the $p_i$. On the other hand, the loci $U_W$, $U_X$, and $U_Y$ are locally products of lines $L_i \cap \mathbb{A}^2$, and accordingly, so are their tangent spaces. Since the each of the lines $L_i$ is transverse to $C_d$ in $\mathbb{P}^2$, it follows that the intersections $\Theta_{3,d} \cap U_W$, $\Theta_{3,d} \cap U_X$, and $\Theta_{3,d} \cap U_Y$ are transverse also. 

\subsection{Tautological bundles}

Given the Chern/Schur classes of tautological bundles coming from line bundles, we can compute the Chern/Schur classes of tautological bundles coming from vector bundles. This follows immediately from three facts. First, given a short exact sequence of vector bundles on $\mathbb{P}^2$, the same sequence where we replace each bundle with the tautological bundle coming from it is a short exact sequence on $\mathbb{P}^{2[n]}$. Second, $\{\mathcal{O}(d),\mathcal{O}(d+1),\mathcal{O}(d+2)\}$ generates the derived category of coherent sheaves on $\mathbb{P}^2$. Third, the ring structure of the Chow ring is $\mathbb{P}^{2[3]}$ is known.  From these three, the computation is a numerical calculation. As the formulas are rather messy, we defer their statement to the Appendix \ref{app a}.

\section{A criterion for nefness}
One immediate application of the computation of these tautological classes, beyond just being interesting in their own right, is to the computation of the ``positive'' cones of cycles of the Hilbert scheme.
In the next two sections, we compute the nef cones in codimension 2 and 3, and by duality, the effective cones in dimensions 2 and 3. Note, both the effective and the nef cones are previously known in (co)dimensions 1 and 5 \cite{LQZ,CHW} so this leaves only nef cone in codimension 4, and its dual the effective cone in codimension 2, to be computed.
In addition to the previous computations, the key idea is to use Kleiman's tranversality theorem on the orbits of the action of $\mathrm{PGL}(3)$ on $\mathbb{P}^{2[3]}$ to provide a criterion for nefness. This reduces the problem to that of understanding the cycle structure (Chow ring, nef cones, and effective cones) of the closure of those orbits. In this section, we introduce the lemmas giving that criterion for nefness.

Let $G$ be an infinite connected group acting on a smooth variety $X$ of dimension $n$ with finitely many orbits $\mathcal{O}_1,\cdots,\mathcal{O}_m$, which are varieties themselves. Let $Y$ be a codimensional $k$ subvariety of $X$ and $Z$ an irreducible $k$-dimensional subvariety of $X$. 
We first state two immediate applications of Kleiman's transversality theorem.

\begin{lem}
If $Y$ intersects an orbit $\mathcal{O}$ of dimension $d$ in the expected dimension, then $gY \cap gZ\vert_\mathcal{O}$ is empty or finitely many points for the general $g\in G$. 
\end{lem}

\begin{proof}
Since $Y$ intersects $\mathcal{O}$ in the expected dimension, we have \[\mathrm{codim}_{\mathcal{O}}(Y \cap \mathcal{O}) + \mathrm{codim}_{\mathcal{O}}(Z \cap \mathcal{O}) \geq  k +(d-k) \geq d.\]
The lemma is then immediate from Kleiman's transversality theorem.
\end{proof}

\begin{lem}
If $Y$ intersects an orbit $\mathcal{O}$ (of dimension $d$) one dimension higher than expected and $\dim(\mathcal{O}\cap Z) < \dim(Z)$, then $gY \cap gZ \vert_\mathcal{O}$ is empty or finitely many points for the general $g\in G$. 
\end{lem}

\begin{proof}
Since $Y$ intersects $\mathcal{O}$ in the expected dimension plus one and  $\dim(\mathcal{O}\cap Z) < \dim(Z)$, we have \[\mathrm{codim}_{\mathcal{O}}(Y \cap \mathcal{O}) + \mathrm{codim}_{\mathcal{O}}(Z \cap \mathcal{O}) \geq  k-1 +(d-k+1) \geq d.\]
The lemma is then immediate from Kleiman's transversality theorem.
\end{proof}

We can piece these lemmas together to give the criterion we wanted for a class to be nef.

\begin{lem}
\label{lem group}
Let $G$ be an infinite group acting on a smooth variety $X$ of dimension $n$ with finitely many orbits $\mathcal{O}_1,\cdots,\mathcal{O}_m$. Let $Y$ be a subvariety on $X$ of codim $k$ which intersects $\mathcal{O}_1, \cdots, \mathcal{O}_l$ in the expected dimension and $\mathcal{O}_{l+1}, \cdots, \mathcal{O}_m$ in at most one dimension higher than expected.  
If $Y$ intersects every $k$-fold contained in the closures of $\mathcal{O}_{l+1}, \cdots, \mathcal{O}_m$ nonnegatively, then $Y$ is nef.
\end{lem}

\begin{proof}
Let $Z$ be an irreducible $k$-dimensional subvariety of $X$. 
It suffices to show that $Y \cdot Z \geq 0$

If $Z \subset \overline{\mathcal{O}_k}$ for some $k$ such that $l+1 \leq k \leq m$, then we have assumed that $Y \cdot Z \geq 0$ as desired.  

If not, then every orbit $\mathcal{O}$ either has that $Y$ intersects it in the expected dimension or that both $Y$ intersects it in the expected dimension plus one and it does not contain an open subset of $Z$.
By the respective lemmas in each case, we have that $gY \cap gZ \vert_\mathcal{O}$ is empty or consists of finitely many points for the general $g \in G$ for all orbits $\mathcal{O}$.
Since $G$ is infinite and there are finitely many orbits, we have that $gY \cap gZ$ is empty or consists of finitely many points for general $g \in G$.
\end{proof}

Note, the lemma is easily seen to apply to classes which have at least one representative for each other interesting that orbit in the stated dimension.

\section{The orbits and their cycle structures}
In this section, we work out the cycle structures of the orbits of the $\text{PGL}(3)$ action on $\mathbb{P}^{2[3]}$ including their nef and effective cones. In the next section, we apply this structure and the lemmas of the previous section to compute the nef and effective cones on the Hilbert scheme. Note that $\text{PGL}(3)$ and its action on $\mathbb{P}^{2[3]}$ satisfy the hypotheses of the criterion from the previous section.

The action has seven orbits. We denote them by $\mathcal{O}$ with a subscript indicating the codimension of the orbit. When there is not a unique orbit in that codimension, we add an additional distinguishing subscript. There is a full-dimensional orbit which consists of collections of three distinct, non-collinear points, denoted by $\mathcal{O}_0$. There are two divisorial orbits. One is the locus of three distinct, collinear points, denoted by $\mathcal{O}_{1,col}$. The other is the locus of non-collinear schemes supported on two distinct points, denoted by $\mathcal{O}_{1,nonred}$. Note, in most literature, the closure of this locus is denoted by $B$ or $E$, designations we avoid for clarity. There are two orbits of codimension two as well. The first is the locus of collinear schemes supported on two distinct points, denoted $\mathcal{O}_{2,col}$. The second is the locus of non-collinear curvilinear schemes supported at a single point, denoted by $\mathcal{O}_{2,nonred}$. Finally, there are one orbit each in codimensions three and four. The codimension three orbit is the locus of collinear schemes supported at a single point, denoted $\mathcal{O}_{3}$. The codimension four orbit is the locus of schemes which are the maximal ideal squared at a point, denoted $\mathcal{O}_{4}$. We denote schemes supported at a single point as totally non-reduced.

In each case, we omit cycle structure in dimensions 0 and the dimension of the locus as they are all trivial. In addition, we compute the classes of the extremal rays of the cones, and of the closures of the orbits, in $A\left( \mathbb{P}^{2[3]}\right)$. In particular, we fix the notation we use to refer to these classes for the rest of the paper.

Also, we say ``the class of the locus where....'', but we do not define the scheme structure on it. 
That is given through the intersection products giving the classes except for the cases of $\mathcal{O}_{2,nonred}$ and $\mathcal{O}_4$ where we compute the class of the reduced structure.

\subsection{The plane of maximal ideals squared}
\label{subsec:M}
First, we have that $\mathcal{O}_{4} = \overline{\mathcal{O}_{4}} \simeq \mathbb{P}^2$. 

In order to compute its class, we first observe that $\overline{\mathcal{O}_{2,nonred}}\cdot F^2$ is supported on two irreducible components, $\mathcal{O}_4$ and $\overline{\mathcal{O}_{2,nonred}}\cdot D$ so $l\mathcal{O}_4 = (F^2-kD)\cdot \overline{\mathcal{O}_{2,nonred}}$.
We first solve for $k$ then for $l$.
Later in this section, we compute that $\overline{\mathcal{O}_{2,nonred}} = 3(A-B+C-D)$.
Since $\overline{\mathcal{O}_{1,col}}$ and $\mathcal{O}_{4}$ are disjoint, $ \mathcal{O}_{4}\cdot \overline{\mathcal{O}_{1,col}}\cdot H =0$.
It is well known that $\overline{\mathcal{O}_{1,col}} = F-H$, and so we have that $\overline{\mathcal{O}_{2,nonred}}\cdot \overline{\mathcal{O}_{1,col}}\cdot H\cdot F^2 = 27$ and $\overline{\mathcal{O}_{2,nonred}}\cdot \overline{\mathcal{O}_{1,col}}\cdot H\cdot D = 9$.
Thus, we have that
\[l \mathcal{O}_{4} = \overline{\mathcal{O}_{2,nonred}}\cdot (F^2-3D) = 9 (\alpha + \delta + \epsilon).\]
In \cite{EL3}, it is computed that $H^2 \cdot \mathcal{O}_4 = 9$. Since $H^2\cdot 9(\alpha+\delta+\epsilon) =27$, $l = 3$.
Thus, \[\mathcal{O}_{4} = \frac{1}{3}\overline{\mathcal{O}_{2,nonred}}\cdot (F^2-3D) = 3 (\alpha + \delta + \epsilon).\]

Since $\mathcal{O}_{4}$ is isomorphic to a (nonreduced) projective plane, its Neron-Severi space is dimension 1. Let $C_{1 }$ be the class of the locus of maximal ideals squared supported on a fixed general line,
\begin{align*}
C_{1 } &= \overline{\mathcal{O}_{4}} \cdot H = 9\psi.\\
\end{align*}
It is immediate that $C_{1 }$ spans the effective and nef cones. 

\subsection{Totally non-reduced and collinear schemes}
\label{subsec:TNR Col}
Second, we have $\mathcal{O}_{3} = \overline{\mathcal{O}_{3}}$. It's class is \[\overline{\mathcal{O}_{3}} = \overline{\mathcal{O}_{2,nonred}}\cdot \overline{\mathcal{O}_{1,col}} =3(3U - 2V - W + 4X - 6Y - Z).\] Note, those two orbit classes will be calculated in their respective subsections. 
This has the structure of a $\mathbb{P}^1$-bundle over $\mathbb{P}^2$. In particular, it is the flag variety of affine lines in affine planes in affine three space, $F(1,2;3)$. That structure, tells us that the Chow ring is generated by two elements in each of dimension one and dimension two.  Let's introduce all the relevant classes. 

For surfaces, let $S_{1}$ be the class of the locus of totally nonreduced schemes collinear with a fixed point.  Similar, let $S_{2}$ be the class of locus of totally nonreduced collinear schemes supported on a fixed line.
This gives 
\begin{align*}
S_{1} &= \overline{\mathcal{O}_{2,nonred}}\cdot A = \frac{1}{3} \overline{\mathcal{O}_3} \cdot F =  3(\beta - \gamma + 2\delta + \epsilon) \text{ and }\\
S_{2} &= H \cdot \overline{\mathcal{O}_{3}} = \overline{\mathcal{O}_{2,nonred}} \cdot \overline{\mathcal{O}_{1,col}}\cdot H= 9(- 2\alpha + \beta).
\end{align*}

For curves, let $C_{2 }$ be the class of the locus of totally nonreduced collinear schemes supported on a fixed general point. 
Similarly, let $C_{3}$ be the class of the locus of totally nonreduced schemes which are contained in a fixed line. 
Then we have that
\begin{align*}
C_{2 } &= \overline{\mathcal{O}_{3}} \cdot E  = 9(\phi-\psi)\text{  and  }\\
C_{3} &= \frac{1}{3} \overline{\mathcal{O}_{3}} \cdot D = \overline{\mathcal{O}_{2,nonred}} \cdot Y= 3(- \phi + 2\psi).
\end{align*}

In both dimensions, the nef and effective cones coincide. 
\begin{prop}
$\text{Nef}^1(\mathcal{O}_3) = \text{Eff}_2(\mathcal{O}_3)$ is spanned by $S_{1}$ and $S_{2}$. Equivalently,  $\text{Eff}_1(\mathcal{O}_3) = \text{Nef}^2(\mathcal{O}_3) $ is spanned by $C_2$ and $C_3$.
\end{prop}
As we are in the (co)dimension 1 case, all of these classes are nef as they are basepoint free/moving, and the duality of these classes follows from the emptyness of the intersections $S_1 \cdot C_3$ and $S_2 \cdot C_2$.

\subsection{Totally non-reduced schemes}
\label{subsec:TNR}
Third, we have $\overline{\mathcal{O}_{2,nonred}}$. To compute its class, we first observe that it intersects $\beta$, $\gamma$, $\delta$, and $\epsilon$ in 0. Thus, $\overline{\mathcal{O}_{2,nonred}} = k(A-B+C-D)$. We now compute the value $k$ in local coordinates.
If $\alpha$ is the locus of schemes containing the origin and contained in the line $\{y=0\}$, then the unique scheme in the intersection is given by the ideal $(x^3,y)$. There are charts, centered at this ideal, given by the coordinates $(a,b,c,d,e,f)$ and the ideals
\[(x^3 + ax^2 + bx + c, y + dx^2 + ex + f).\]
In this chart, $\alpha$ is defined by the equations $c = d = e = f = 0$, with $c=f=0$ coming from all the schemes in the locus containing the origin and $d = e = f = 0$ coming from all schemes being contained in the line $y = 0$. The totally nonreduced locus is the set of ideals with generator of the form $(x-A)^3$ for any $A \in \mathbb{C}$, and is therefore given by the equations
\[3b = a^2 \text{ and } 27c= a^3.\]
The ring of the intersection is therefore
\[\frac{\mathbb{C}[a,b,c,d,e,f]}{(3b - a^2, 27c - a^3, c,d,e,f)} \equiv \frac{\mathbb{C}[a]}{(a^3)}\]
which has basis given by $1,a,a^2$, and is therefore of length three. This gives the class
\[\overline{\mathcal{O}_{2,nonred}} = 3(A-B+C-D).\]

Next, $\overline{\mathcal{O}_{2,nonred}}$ has the structure of a $\mathbb{P}^2$ bundle over the punctual Hilbert scheme of three points, $\mathbb{A}_0^{2[3]}$. 
As $\mathbb{A}_0^{2[3]}$ is isomorphic to the quadric cone, we know it's Chow ring is dimension three. 
Then we know that the Chow ring of $\overline{\mathcal{O}_{2,nonred}}$ has dimension nine. 
Furthermore, we know that there are two dimensions of curve and threefold and three dimensions of surface.  

For threefolds, let $T_{1}$ be the class of the locus of totally nonreduced schemes supported on a fixed line. Similarly, let $T_{2}$ be the class of the locus of totally nonreduced collinear schemes. Then we have that
\begin{align*}
T_{1} &= \overline{\mathcal{O}_{2,nonred}}\cdot H = 3(W - 3X + 3Y)\text{ and }\\
T_{2} &=   \overline{\mathcal{O}_{2,nonred}}\cdot \overline{\mathcal{O}_{1,col}}= \overline{\mathcal{O}_{3}} = 3(3U - 2V - W + 4X - 6Y - Z).
\end{align*}

For surfaces, let $S_{3}$ be the class of the locus of totally nonreduced schemes supported at a single fixed point. Similarly, let $S_{4}$ be the locus of schemes which have length two supported on a fixed line and are supported at a single (moving) point. 
Then we have that
\begin{align*}
S_{3} &= \overline{\mathcal{O}_{2,nonred}}\cdot E = 3(\alpha - \beta + \gamma - \delta)\text{ and }\\
S_{4} &= \overline{\mathcal{O}_{2,nonred}}\cdot D = 3(-2\alpha+\beta).\\
\end{align*}

A multiple ($\frac{2}{3}$) of already mentioned curve $C_{1 }$ is also represented by the locus of totally nonreduced schemes contained in a general smooth conic.
We can see this locus is a multiple of $C_{1 }$ as they are both dual to the divisor of collinear schemes both in the orbit and in the Hilbert scheme.
Similarly, $C_{2 }$ is also represented by the locus of schemes supported at a point with a length 2 subscheme collinear with a general fixed point.
We can see this locus must be at least a multiple of $C_{2 }$ as they both intersect $H$ and its restriction to the orbit in zero.
Together this gives that
\begin{align*}
\tfrac{2}{3}C_{1 } &= \overline{\mathcal{O}_{2,nonred}}\cdot c_3\left(\mathcal{O}(2)^{[3]}\right) = 6 \psi \text{ and}\\
C_{2 } &= \overline{\mathcal{O}_{2,nonred}}\cdot E \cdot F = 9(\phi-\psi).
\end{align*}

Recall that the the class $\overline{\mathcal{O}_{4}}$ was computed in Subsection \ref{subsec:M}.

\begin{prop}
$\text{Nef}^1(\overline{\mathcal{O}_{2,nonred}})$ is spanned by $T_{1}$ and $T_{1}+T_{2}$, $\text{Nef}^2(\overline{\mathcal{O}_{2,nonred}})$ is spanned by $S_{3}$,  $S_{4}$, and $\mathcal{O}_{4}$, and
$\text{Nef}^3(\overline{\mathcal{O}_{2,nonred}})$ is spanned by $C_{1 }$ and $C_{2 }$.
Conversely, $\text{Eff}_1(\overline{\mathcal{O}_{2,nonred}})$ is spanned by $C_{2 }$ and $C_{3}$, $\text{Eff}_2(\overline{\mathcal{O}_{2,nonred}})$ is spanned by $S_{1}$,  $S_{3}$, and $S_{4}$, and
$\text{Eff}_3(\overline{\mathcal{O}_{2,nonred}})$ is spanned by $T_{1}$ and $T_{2}$. 
\end{prop}

We have given effective representatives of all classes and the duality of the cones are clear from the appropriate empty intersections other than $C_{3 } \dot (T_1 +T_2)$ and $\mathcal{O}_4 \cdot S_4$ so it remains to show those and that the nef classes are, in fact, nef.
Those intersections follow as $C_{3 } \dot (T_1 +T_2) = C_{3 } \cdot (H+ \mathcal{O}_{1,col})$ and $\mathcal{O}_4 \cdot S_4 = \mathcal{O}_4 \cdot D$ where the first intersections in each equality are taken in $\overline{\mathcal{O}_{2,nonred}}$ and the second ones are taken in $\mathbb{P}^{2[3]}$.
For the nef classes being nef, we first note that the action on the Hilbert scheme restricts to an action on this orbit closure.
Then $T_{1}$, $S_{3}$, $C_{1 }$, and $C_{2 }$ intersect both non-open suborbits in the correct dimensions so are nef. Note, showing this for $C_{2 }$ requires using both representatives we have given.

Lemma  \ref{lem group} also applies to $S_{4}$ as it intersects $\overline{\mathcal{O}_3}$ in the correct dimension but intersects $\overline{\mathcal{O}_4}$ in a curve. Since $\mathcal{O}_4 \cdot S_{4} = \frac{1}{3} \overline{\mathcal{O}_{2,nonred}} \cdot (F^2-3D) \dot D = 2$, $S_{4}$ is nef by Lemma \ref{lem group}. 
$\overline{\mathcal{O}_4}$ intersects $\overline{\mathcal{O}_3}$ in the correct dimension (in fact they are disjoint), but obviously intersects itself incorrectly.
However, since it can be seen numerically that $\mathcal{O}_4 = S_1+S_3$ in both the orbit and the whole Hilbert scheme, there are representatives of the class which intersect $\mathcal{O}_4$ correctly so it is nef by Lemma \ref{lem group}.

This leaves $T_{1}+T_{2}$ to be shown to be nef.
The obvious irreducible representatives of $T_{1}+T_{2}$ intersect $\overline{\mathcal{O}_4}$ in the correct dimension but contain $\overline{\mathcal{O}_3}$. 
If a curve intersects $\overline{\mathcal{O}_3}$ in points (or is disjoint from it), then $T_{1}+T_{2}$ intersects it nonnegatively. 
If not, the curve is contained in $\overline{\mathcal{O}_3}$ so it a positive linear combination of that orbits extremal effective curves.
In Subsection \ref{subsec:TNR Col}, we see that the effective cone of that orbit is spanned by $C_{1 }$ and $C_{3 }$. Then $(T_{1}+T_{2}) \cdot C_{1 } = 9$ and $(T_{1}+T_{2}) \cdot C_{3} = 0$. Thus, $T_{1}+T_{2}$ is nef.

\subsection{Collinear and non-reduced schemes}
Fourth, we have $\overline{\mathcal{O}_{2,col}}$. It's class is \[\overline{\mathcal{O}_{2,col}} = \overline{\mathcal{O}_{1,col}} \cdot \overline{\mathcal{O}_{1,nonred}} = (F-H)\cdot 2(2H-F) = 2(- 3A + 2B - 2C + 4D + 2E ).\] 
Note the classes of $\overline{\mathcal{O}_{1,col}}$ and $\overline{\mathcal{O}_{1,nonred}}$ are well known.

It has the structure of a $\mathbb{P}^1$-bundle over the non-reduced locus in $\mathbb{P}^{2[2]}$. You can also see this directly as a $\mathbb{P}^1 \times \mathbb{P}^1$ bundle over $\mathbb{P}^{2*}$. From either of these descriptions, we see that the Chow ring has three dimensions of curve and threefold and four dimensions of surface.  Let us start by listing all of the relevant classes.

For threefolds, let $T_{3}$ be the class of the locus of nonreduced schemes collinear with a fixed point. Let $T_{4}$ be the class of the locus of nonreduced collinear schemes with the reduced point supported on a fixed line. Finally, let $T_{5}$ by the class of the locus of nonreduced collinear schemes with the nonreduced point supported on a fixed line. 

In order to compute the classes of $T_{4}$ and $T_{5}$, we need to introduce two fourfolds and a surface and compute their classes.
First, let $S_{5}$ be the class of the locus containing a fixed point and nonreduced point supported on a fixed line containing the fixed point.
Then, we have the emptiness of the intersections of $S_{5}$ with $A$, $B$, $C$, and $E$. For some $k>0$ (which we determine later), we then have that
\[S_{5} = k(-\alpha +\beta-\delta).\]

Let $F_{1}$ be the class of the closure of the locus of schemes with a freely moving nonreduced point and a reduced point on a fixed general line. Similarly, let $F_{2}$ be the class of the closure of the locus of schemes with a freely moving reduced point and a nonreduced point supported on a fixed general line. Then the computations of the classes of $F_{1}$ and $F_{2}$ follow from the equalities $\overline{\mathcal{O}_{1,nonred}}\cdot H = F_{1}+F_{2}$, $\overline{\mathcal{O}_{1,nonred}}\cdot H\cdot \beta = 2$, and $\overline{\mathcal{O}_{1,nonred}}\cdot H \cdot \delta = 2$ and the emptiness of the following seven intersections: $F_{1} \cdot \gamma$, $F_{1} \cdot \delta$, $F_{1} \cdot \epsilon$, $F_{2} \cdot \beta$, $F_{2} \cdot \gamma$, $F_{2} \cdot \epsilon$, and $F_{1} \cdot S_{5}$.
From this, we get
\begin{align*}
F_{1} &= 2(-B+C) \text{  and  } \\
F_{2} & = 2(C-2D).\\
\end{align*}
This finally gives us
\begin{align*}
T_{3} &=  F \cdot \overline{\mathcal{O}_{2,col}} = \overline{\mathcal{O}_{1,nonred}}\cdot A =  6(- 2U + 2V + 3Y + 2Z), \\
T_{4} &= \overline{\mathcal{O}_{1,col}} \cdot F_{1} = 2(2U - V - W + 4X - 3Y)\text{,  and  } \\
T_{5} & = \overline{\mathcal{O}_{1,col}} \cdot F_{2} = 2(U - W + 4X - 6Y + Z).\\
\end{align*}

For surfaces, let $S_{6}$ be the class of the locus of nonreduced collinear schemes contained in a fixed line. Let $S_{7}$ be the class of the locus of nonreduced collinear schemes whose reduced point is at a fixed point. Let $S_{8}$ be the class of the locus of nonreduced collinear schemes whose reduced point is on a fixed line and whose non-reduced point is supported on a different fixed line. Finally, let $S_{9}$ be the class of the locus of collinear schemes containing a subscheme of length two supported at a fixed point.

In order to compute $S_{7}$ and $S_{9}$ we need to introduce two threefolds. Let $T_{6}$ be the class of the locus of schemes containing a fixed point and a freely moving nonreduced point. Similarly, let $T_{7}$ be the class of the locus of schemes containing a nonreduced fixed point and a freely moving point.
The emptiness of the following ten intersections computes those threefolds up to multiple: $T_{6} \cdot U$, $T_{6} \cdot V$, $T_{6} \cdot W$, $T_{6} \cdot X$, $T_{6} \cdot Y$, $T_{7} \cdot U$, $T_{7} \cdot W$, $T_{7} \cdot X$, $T_{7} \cdot Y$, and $T_{7} \cdot Z$. Finally, knowing that $ \overline{\mathcal{O}_{1,nonred}} \cdot E= T_{6} +T_{7}$ gives that
\begin{align*}
T_{6} &= 2(U-V) \text{ and }\\
T_{7} &= 2(U-Z). \\
\end{align*}
Finally, we have that 
\begin{align*}
S_{6} &= D \cdot \overline{\mathcal{O}_{2,col}} = 12 \alpha, \\
S_{7} &= \overline{\mathcal{O}_{1,col}} \cdot T_{6}= 2(-\alpha+\beta-\gamma+2\delta+\epsilon), \\
S_{8} &= \frac{1}{2} C \cdot \overline{\mathcal{O}_{2,col}} = - 6\alpha + 7\beta - \gamma + 2\delta + \epsilon \text{,  and  }\\
S_{9} &= \overline{\mathcal{O}_{1,col}} \cdot T_{7}= 2(- 2\alpha + \beta - \gamma + 2\delta + \epsilon).
\end{align*}

For curves, let $C_{4}$ be the class of the locus of nonreduced collinear schemes contained in a fixed line with their reduced point at a fixed point. 
Let $C_{5}$ be the class of the locus of nonreduced collinear schemes contained in a fixed line with their non-reduced point supported at a fixed point. 
Let $C_{6}$ be the class of the locus of nonreduced collinear schemes whose reduced point is at a fixed general point and whose non-reduced point is supported on a fixed line. 
Finally, let $C_{7}$ be the class of the locus of nonreduced collinear schemes whose nonreduced point is supported at a fixed general point and whose reduced point is supported on a fixed line.

In order to compute $C_{4}$ and $C_{5}$, we need to introduce two surface classes. Let $S_{10}$ be the class of the locus of schemes which contain a nonreduced point supported at a fixed point and a point on a line through that point. Let $S_{11}$ be the class of the locus of schemes containing a fixed double point.
We can then compute these classes up to multiple using the emptiness of the following intersections: $S_{10} \cdot A$, $S_{10} \cdot C$, $S_{10} \cdot D$, $S_{10} \cdot E$, $S_{11} \cdot A$, $S_{11} \cdot B$, $S_{11} \cdot C$, $S_{11} \cdot D$.
Knowing that $\overline{\mathcal{O}_{1,nonred}} \cdot Z = S_{5}+S_{10}+S_{11}$ gives the classes of the surfaces and the multiple of $S_{5}$:
\begin{align*}
S_{5} &= 2(-\alpha+\beta-\delta), \\
S_{10} &= 2(-\alpha+\delta)  \text{, and }\\
S_{11} &= 2\epsilon.
\end{align*}
Putting this together finally gives
\begin{align*}
C_{4} &= S_{5}\cdot \overline{\mathcal{O}_{1,col}}= 4(-\phi+2\psi), \\
C_{5} &= S_{10}\cdot \overline{\mathcal{O}_{1,col}}= 2(-\phi+2\psi), \\
C_{6} &= \overline{\mathcal{O}_{1,col}} \cdot H \cdot T_{6} = 2(\phi+\psi) \text{, and }\\
C_{7} &= \overline{\mathcal{O}_{1,col}} \cdot H \cdot T_{7} = 2(2\phi-\psi).\\
\end{align*}

\begin{prop}
$\text{Nef}^1\left( \overline{\mathcal{O}_{2,col}}\right)$ is spanned by $T_{3}$, $T_{4}$, and $T_{5}$, $\text{Nef}^2\left( \overline{\mathcal{O}_{2,col}}\right)$ is spanned by $S_{6}$, $S_{7}$, $S_{8}$, and $S_{9}$, and $\text{Nef}^3\left( \overline{\mathcal{O}_{2,col}}\right)$ is spanned by $C_{4}$, $C_{5}$, $C_{6}$, and $C_{7}$.
Conversely, $\text{Eff}_1\left( \overline{\mathcal{O}_{2,col}}\right)$ is spanned by $C_{4}$, $C_{5}$, and $C_{2 }$, $\text{Eff}_2\left( \overline{\mathcal{O}_{2,col}}\right)$ is spanned by $S_{2}$, $S_{6}$, $S_{7}$ and the $S_{9}$, and $\text{Eff}_3\left( \overline{\mathcal{O}_{2,col}}\right)$ is spanned by $T_{2}$, $T_{3}$, $T_{4}$, and $T_{5}$.
\end{prop}

Again, we just need to show that the proposed nef classes are actually nef as the dualities of the cones are clear from the emptiness of the appropriate intersections, but that follows immediately from Lemma \ref{lem group} as all of the nef extremal rays intersect the only proper suborbit in the correct dimension.

\subsection{Collinear schemes}
Fifth, we have $\overline{\mathcal{O}_{1,col}}$. The class of $\overline{\mathcal{O}_{1,col}}$ is again well known and is \[\mathcal{O}_{1,col} = F-H.\] It has the structure of a $\mathbb{P}^{1[3]} \simeq \mathbb{P}^3$ bundle over $\mathbb{P}^{2*} \simeq \mathbb{P}^2$. As such, it has two dimensions of curve and fourfold and three dimensions of surface and threefold. Let us start by listing all of the relevant classes.
For fourfolds, let $F_{3}$ be the class of the locus of three points collinear with a fixed point. Let $F_{4}$ be the class of the locus of subschemes that are collinear and incident to a fixed line. Finally let $F_{5}$ be the class of the locus of nonreduced collinear schemes. Then we can see that
\begin{align*}
F_{3} &= \frac{1}{3} \cdot F \cdot \overline{\mathcal{O}_{1,col}} = A, \\
F_{4} &=  H \cdot \overline{\mathcal{O}_{1,col}} = B - C  + 2D + E,  \text{  and  }\\
F_{5} &= \overline{\mathcal{O}_{1,nonred}} \cdot \overline{\mathcal{O}_{1,col}} = \overline{\mathcal{O}_{2,col}}= 2(- 3A + 2B - 2C + 4D + 2E).
\end{align*}

For threefolds, let $T_{8}$ be the class of the locus of schemes contained in a fixed line. Let $T_{9}$ be the class of the locus of collinear schemes incident to each of a pair of distinct fixed lines. Let $T_{10}$ be the class of the locus of collinear schemes containing a fixed point.  Then we have that 
\begin{align*}
T_{8} &= \frac{1}{9}\cdot F^2 \cdot \overline{\mathcal{O}_{1,col}} = Y, \\
T_{9} &= C \cdot \overline{\mathcal{O}_{1,col}} = U - W + 4X + Z, \\
T_{10} &= E \cdot \overline{\mathcal{O}_{1,col}} = - U + V + Z \text{,  and  }\\
\end{align*}

For surfaces, let $S_{12}$ be the class of the locus of schemes contained in a fixed line and containing a fixed point.  Let $S_{13}$ be the class of the locus of collinear schemes containing a fixed point and incident to a general fixed line. 
Finally, let $S_{14}$ be the class of the locus of collinear schemes incident to each of a set of three general lines. 

Then we have that 
\begin{align*}
S_{12} &= \frac{1}{9}\cdot H \cdot F^2 \cdot \overline{\mathcal{O}_{1,col}} = \alpha, \\
S_{13} &= H \cdot E \cdot \overline{\mathcal{O}_{1,col}} = \beta - \gamma + 2\delta + \epsilon\text{, and }\\
S_{14} &= W \cdot \overline{\mathcal{O}_{1,col}} = 3\beta.\\
\end{align*}

For curves, let $C_{8 }$ the class of the locus of schemes contained in a fixed line and containing two distinct fixed points.
\begin{align*}
C_{8 } &= \alpha \cdot H =  \frac{1}{9}\cdot \overline{\mathcal{O}_{1,col}} \cdot H^2\cdot F^2= -\phi + 2\psi 
\end{align*}


\begin{prop}
$\text{Nef}^1\left(\overline{\mathcal{O}_{1,col}}\right)$ is spanned by $F_{3}$ and $F_{4}$, $\text{Nef}^2\left(\overline{\mathcal{O}_{1,col}}\right)$ is spanned by  $T_{8}$, $T_{9}$, and $T_{10}$, $\text{Nef}^3\left(\overline{\mathcal{O}_{1,col}}\right)$ is spanned by $S_{12}$, $S_{13}$, and $S_{14}$, and $\text{Nef}^4\left(\overline{\mathcal{O}_{1,col}}\right)$ is spanned by $C_{8 }$ and $C_{6}+C_{7}$.
Conversely, $\text{Eff}_1\left(\overline{\mathcal{O}_{1,col}}\right)$ is spanned by $C_{2 }$ and $C_{8 }$, $\text{Eff}_2\left(\overline{\mathcal{O}_{1,col}}\right)$ is spanned by $S_{2}$, $S_{9}$, and $S_{12}$, $\text{Eff}_3\left(\overline{\mathcal{O}_{1,col}}\right)$ is spanned by $T_{2}$, $T_{8}$, and $T_{10}$, and $\text{Eff}_4\left(\overline{\mathcal{O}_{1,col}}\right)$ is spanned by $F_{3}$ and $F_{5}$,
\end{prop}

The duality of these cones is again to the emptiness of the appropriate intersections except in the case of $F_5 \cdot (C_{6 } + C_{7 })$ where it follows from that intersection in this orbit equaling the intersection $\overline{\mathcal{O}_{1,nonred}}\cdot (C_{6 } + C_{7 })$ in the full Hilbert scheme.
Then it suffices to justify why each extremal ray of the nef cones are nef. 
All of the relevant classes intersect both suborbits of the action in the correct dimension so are nef by Lemma \ref{lem group} except for $C_6+C_7$. 
The lemma applies to this curve as it is disjoint from $\overline{\mathcal{O}_3}$ and intersects $\overline{\mathcal{O}_{2,col}}$ in a curve. 
Since $\overline{\mathcal{O}_{2,col}} \cdot (C_6+C_7) = 0$, it is nef.

\subsection{Non-reduced schemes}
Sixth, we have $\overline{\mathcal{O}_{1,nonred}}$.
Its class is well known and is \[\overline{\mathcal{O}_{1,nonred}} =2(2H-F).\]
Every cycle we need intersects this orbit in the correct dimension, so we omit its cycle structure.

\subsection{The open orbit}
Finally, we have $\overline{\mathcal{O}_0} = \mathbb{P}^{2[3]}$, whose Chow ring was recalled in Section \ref{sec:msbasis}. This is the space we are studying.

\section{The nef and effective cones}
In this section, we compute the nef (effective) cones in codimensions 2 and 3 (dimension 2 and 3).


\begin{thm}
$\mathrm{Eff}_2\left(\mathbb{P}^{2[3]}\right)$ is spanned by the following six classes: $\alpha$, $\epsilon$, $-\alpha+\delta$, $-\alpha+\beta-\delta$, $\alpha-\beta+\gamma-\delta$, and $-2\alpha+\beta-\gamma+2\delta+\epsilon$.\\
Conversely, $\mathrm{Nef}^2\left(\mathbb{P}^{2[3]}\right)$ is spanned by the following six classes: $B$, $C$, $D$, $E$, $A+B$, and $A+E$.
\end{thm}

\begin{proof}
$C$ and $E$ intersect each orbit in the correct dimension so are nef by Lemma \ref{lem group}.

$B$ and $D$ intersect $\mathcal{O}_0$, $\mathcal{O}_{1,nonred}$, $\mathcal{O}_{1,col}$, $\mathcal{O}_{2,col}$, $\mathcal{O}_{2,nonred}$, and $\mathcal{O}_{3}$ in the correct dimension and intersect $\mathcal{O}_{4}$ in one dimension higher than they should. 
The only surface in $\mathcal{O}_{4}$  is itself and $\mathcal{O}_4 \cdot D =0$ and $\mathcal{O}_4 \cdot B =3$ so $B$ and $D$ are nef by the lemma.

$A$ by itself intersects $\mathcal{O}_{1,nonred}$, $\mathcal{O}_{2,nonred}$, and $\mathcal{O}_{4}$ in the correct dimension and $\mathcal{O}_{1,col}$, $\mathcal{O}_{2,col}$, and $\mathcal{O}_{3}$ in one dimension higher than it should. 

If follows that $A+B$, and $A+E$ both intersect every orbit in at most one dimension greater than expected and intersect $\mathcal{O}_{1,nonred}$ in correct dimension. Pairing them against the extremal cycles of the other orbits which we computed in the previous section, we see that they are all nef by the lemma.

We have previously constructed effective representatives of $\alpha$, $\epsilon$, $\alpha-\beta+\gamma-\delta$, $-\beta+4\delta+\epsilon$, and $-2\alpha+\beta-\gamma+2\delta+\epsilon$.
Similarly, we have given effective representatives of multiples of  $-\alpha+\delta$ and $-\alpha+\beta-\delta$ so all of the effective classes are in fact, effective.

The duality of the cones these generate is a numeric calculation left as an exercise to the reader.
\end{proof}

\begin{thm}
$\mathrm{Eff}_3\left(\mathbb{P}^{2[3]}\right)$ is spanned by the following seven classes: $Y$, $-U+V+Z$, $3U-2V-W+4X-6Y-Z$, $W-3X+3Y$,  $X-3Y$, $U-V$, and  $U-Z$.\\
Conversely, $\mathrm{Nef}^3\left(\mathbb{P}^{2[3]}\right)$ is spanned by the following eight classes: $U$, $V$, $W$, $X$, $Z$, $V+Y$, $X+Y$, and $2Y+Z$.
\end{thm}

\begin{proof}
$U$, $V$, $W$, $X$, and $Z$ intersect each orbit in the correct dimension. 

$Y$ by itself intersects $\mathcal{O}_0$, $\mathcal{O}_{1,nonred}$, $\mathcal{O}_{2,nonred}$, and $\mathcal{O}_{4}$ in the correct dimension and intersects $\mathcal{O}_{1,col}$, $\mathcal{O}_{2,col}$, and $\mathcal{O}_{3}$ in one dimension higher than it should.

If follows that $U$, $V$, $W$, $X$, $Z$, $V+Y$, $X+Y$, and $2Y+Z$ all all intersect every orbit in at most one dimension greater than expected and all intersect $\mathcal{O}_{1,nonred}$ in correct dimension. Pairing them against the extremal cycles of the other orbits which we computed in the previous section, we see that they are all nef by the lemma.

We have previously constructed effective representatives of $Y$ and $-U+V+Z$ and effective represetatives of multiples of $3U-2V-W+4X-6Y-Z$, $W-3X+3Y$, $U-V$ and $U-Z$.
Up to multiple, an effective representative of the class $X-3Y$ is given by the locus of schemes that are non-reduced and supported on a fixed line.
One can see this class through the empty intersection of this locus with $U$, $V$, $W$, $Y$, and $Z$.
That this locus give that three class is a computation in local coordinates as in the computation of the class of $\overline{\mathcal{O}_{2,nonred}}$.

The duality of the two cones is a numerical calculation which is left as an exercise to the reader. 
\end{proof}



\section{The pliant cone via 2-very ampleness}
In this section, we give sufficient conditions for the 2-very ampleness of vector bundles.
The Schur classes of the tautological bundle on $\mathbb{P}^{2[3]}$ of a 2-very ample vector bundle on $\mathbb{P}^2$ are pliant so we then apply those conditions to construct inner bounds on the pliant cones of cycles.

We first have the following proposition giving a sufficient condition for 2-very ampleness using short exact sequences.

\begin{prop}\label{Prop: cokernel ampleness}
Let $V$ be a semi-stable vector bundle defined by the short exact sequence \[0 \to A \to B \to V \to 0\] where $h^2(A) = h^2(A\otimes I_Z) = 0$ and $B$ is a vector bundle with no higher cohomology that is k-very ample. Then $V$ is k-very ample. 

In particular, the proposition holds if $\mu(A) > -3$, $B$ is 2-very ample, and $B$ is general in its moduli space.
\end{prop}

The proposition follows immediately from the long exact sequence in cohomology.
Using this proposition, we can charactersize $n$-very ampleness for exceptional vector bundles.
\begin{prop}\label{prop: exc ampleness}
An exceptional vector bundle is $n$-very ample iff its slope is at least $n$.
\end{prop}

\begin{proof}
The only if part of the statement is immediate from the fact that if you are $n$-very ample your determinant is nef on $\mathbb{P}^{2[n-1]}$, the computation of that nef cone in \cite{LQZ}, and the known computation of the determinant.
For the if part of the proof, we work by induction on the power of 2 in the denominator of the dyadic integer corresponding to the exceptional bundle $E$, denoted by $q$.
If $q=0$, $\epsilon(k) = k$ and $E_k = \mathcal{O}(k)$ for which the result is well know. 
If $q>1$, then recall that \[\epsilon \left( \frac{p}{2^{q}}\right) = \epsilon \left( \frac{p-1}{2^{q}}\right). \epsilon \left( \frac{p+1}{2^{q}}\right)\] and $E$ is the middle bundle of exceptional collection $\{\alpha, \alpha.\beta, \beta\}$ where $\alpha = \epsilon \left( \frac{p-1}{2^{q}}\right)$, $\beta = \epsilon \left( \frac{p+1}{2^{q}}\right)$, and $\alpha . \beta = \epsilon \left( \frac{p}{2^{q}}\right)$.
Note, since $p$ was odd and $\epsilon(\frac{p}{2^{q}}) > \epsilon(n) = n$, $\frac{p - 1}{2^q}$ reduces and is greater than or equal to $n$ (so $\alpha \geq n$).
By induction hypothesis then, the conclusion holds for $E_\alpha$.

We then have the exceptional collection $\{\beta-3,\alpha,\alpha .\beta\}$.
From this collection, we know that we have the short exact sequence 
\[0 \to E_{(\beta-3).(\alpha.\beta)} \to E_{\alpha}^{\chi(E_{\alpha}.E_{\alpha.\beta})}\to E_{\alpha.\beta} \to 0.\]
Since $-3 < \beta - 3 < (\beta-3).(\alpha.\beta)$, $h^2\left(E_{(\beta-3).(\alpha.\beta)} \right) =h^2\left(E_{(\beta-3).(\alpha.\beta)} \otimes I_Z \right) =0$ for all ideal sheaves $I_Z$ of $n+1$ points.
Since the induction hypothesis holds for $E_{\alpha}$ and $\chi(E_{\alpha}) >0$ as $\alpha>0$, we can then apply Proposition \ref{Prop: cokernel ampleness} to get that $E_{\alpha.\beta}$ is $n$-very ample as desired.
\end{proof}

For non-exceptional bundles, we recall that the general element $V$ of the moduli space of semi-stable sheaves with chern character $\xi$ has a Gaeta-type resolution of one of the forms \[0 \to \mathcal{O}(d)^a \bigoplus \mathcal{O}(d+1)^b \to \mathcal{O}(d+2)^c \to V \to 0 \text{ or}\]
\[0 \to \mathcal{O}(d)^a \to \mathcal{O}(d+1)^b \bigoplus \mathcal{O}(d+2)^c \to V \to 0\] where $a$, $b$, $c$, and $d$ are all determined numerically. 
The form of the resolution and $d$ are determined by which region of the $(\mu,\Delta)$-plane $\xi$ lies in. These regions split up the upper plane using parabolas of the form $\frac{1}{2}(\mu-d)(\mu-d+1)$ and $\frac{1}{2}(\mu-d-1)(\mu-d+1)$ as shown in the figure.
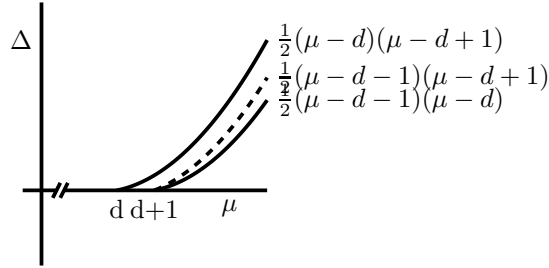
\[\begin{tikzpicture}[domain=0:2, scale = 0.5]
\draw[black, line width = 0.50mm]   plot[smooth,domain=-.5:.4] (\x, {0});
\draw[black, line width = 0.50mm]   plot[smooth,domain=.6:6] (\x, {0});
\draw[black, line width = 0.50mm]   plot[smooth,domain=.3:.5] (\x, {2*(\x-.4)});
\draw[black, line width = 0.50mm]   plot[smooth,domain=.5:.7] (\x, {2*(\x-.6)});
\draw[black, line width = 0.50mm]  (0,-2)--(0,5);
\draw[black, line width = 0.50mm]   plot[smooth,domain=2:6] (\x, {.2*(\x-2)*(\x-1)});
\draw[black,dashed, line width = 0.50mm]   plot[smooth,domain=3:6] (\x, {.2*(\x-3)*(\x-1)});
\draw[black, line width = 0.50mm]   plot[smooth,domain=3:6] (\x, {.2*(\x-3)*(\x-2)});
\node[below] at (2, 0)  {d};
\node[below] at (3, 0)  {d+1};
\node[left] at (0,4)  {$\Delta$};
\node[below] at (5,0)  {$\mu$};
\node[right] at (6,4)  {$\frac{1}{2}(\mu-d)(\mu-d+1)$};
\node[right] at (6,3)  {$\frac{1}{2}(\mu-d-1)(\mu-d+1)$};
\node[right] at (6,2.4)  {$\frac{1}{2}(\mu-d-1)(\mu-d)$};
\end{tikzpicture}\]
The general bundles $V$ for which $(\mu(V),\Delta(V))$ falls in the region between $\frac{1}{2}(\mu-d)(\mu-d+1)$ and $\frac{1}{2}(\mu-d-1)(\mu-d+1)$ has a resolution of the first form while one which falls between $\frac{1}{2}(\mu-d-1)(\mu-d+1)$ and $\frac{1}{2}(\mu-d)(\mu-d+1)$ has a resolution of the second form.

If $d\leq -3$, it is immediate that $V$ has no global sections and therefore cannot be $2$-very ample. 
If $d \geq 1$, all of the line bundles in both possible middle terms are $2$-very ample so the proposition tells us that $V$ is $2$-very ample.
The same is true if $d \geq 0$ and you have the first form of resolution (including the case $b=0$).
Thus, the only uncertain cases are the cases $d=-2$, $d=-1$, and the second form for $d=0$. 
An exact classification of 2-very ampleness is certainly possible, but goes beyond the scope of this paper as it is unknown if it is equivalent to computing the pliant cone. 

\subsection{The pliant cone}
When a bundle $V$ is 2-very ample, the Schur classes of $V^{[3]}$ are pliant as they are the pullbacks of the Schubert classes from the morphism to a Grassmannian given by that tautological bundle. 
We can then use the vector bundles we have shown to be pliant to generate an inner bound on the pliant cone. 
We have shown three types of bundle to be 2-very ample, which we now recall. 

The first type is all the exceptional bundles with slope between 2 and 3.
The second type is the limit of the exceptional bundles as we twist by arbitrarily large multiples of $\mathcal{O}(1)$.
Note, we only use those bundles with rank less than 100.
By \cite{DLP}, these are $\mathcal{O}(d)$, $T_{\mathbb{P}^2}(d)$, $E_{\frac{2}{5}}(d)$, $E_{\frac{3}{5}}(d)$, $E_{\frac{5}{13}}(d)$, $E_{\frac{8}{13}}(d)$, $E_{\frac{12}{29}}(d)$, $E_{\frac{17}{29}}(d)$, $E_{\frac{13}{34}}(d)$, $E_{\frac{21}{34}}(d)$, $E_{\frac{34}{89}}(d)$, and $E_{\frac{55}{89}}(d)$.
The third type is bundles along the parabola given by \[0 \to \mathcal{O}^a \to \mathcal{O}(2)^c \to V \to 0.\]
As we are only looking at the limits of bundles along the parabola, and all such curves are parabolas with the same leading coefficient of $\frac{1}{2}$ (i.e. $\Delta \sim \frac{1}{2} \mu^2$), it does not matter which one we chose. 

The Schur classes of these generate the following cones which are inner bounds on the pliant cone. 

\begin{prop}
$\text{Pl}^2\left( \mathbb{P}^{2[3]}\right)$ contains the cone spanned by the classes $C$, $C+2E$, $A+B+D$, $2A+D+2E$, $2A+B+4D+7E$, $5A+5B+C+7D+5E$, 
$7A+3B+9D+12E$, 
$35A+30B+6C+52D+50E$, 
$40A+25B+3C+58D+69E$, 
$247A+195B+36C+369D+384E$, 
$260A+182B+28C+385D+434E$, 
$1717A+1309B+231C+2563D+2739E$, 
$1751A+1275B+210C+2605D+2870E$, 
$11837A+8900B+1540C+17656D+19052E$, and
$11926A+8811B+1485C+17766D+19395E$.\\
$\text{Pl}^3\left( \mathbb{P}^{2[3]}\right)$ contains the cone spanned by the classes 
$2Y+Z$,
$X+Y$,
$W$,
$W+6X+3Y$,
$2V+X+7Y+2Z$,
$5V+2X+12Y+10Z$,
$5U+10V+W+24X+32Y+10Z$,
$6U+W$,
$6U+7V+10X+15Y+8Z$,
$8U+30V+35X+85Y+38Z$,
$24U+90V+2W+132X+276Y+105Z$,
$69U+390V+2W+447X+1203Y+501Z$,
$96U+230V+10W+368X+643Y+276Z$, and
$648U+3250V+35W+4078X+10093Y+4028Z$.\\
$\text{Pl}^4\left( \mathbb{P}^{2[3]}\right)$ contains the cone spanned by the classes 
$\gamma$,
$\gamma+\epsilon$,
$\alpha+\delta+\epsilon$,
$\alpha+\beta+\delta$,
$2\alpha+\beta+\epsilon$,
$2\alpha+\beta+4\delta+4\epsilon$,
$3\alpha+5\beta+2\gamma+4\epsilon$,
$5\alpha+4\beta+4\delta+4\epsilon$,
$7\alpha+2\beta+2\delta+5\epsilon$, and
$8\alpha+8\beta+2\gamma+12\delta+5\epsilon$.
\end{prop}

\appendix
\section{Chern classes of tautological bundles}
\label{app a}

In this appendix, we list all Chern classes of tautological bundles coming from any vector bundle on $\mathbb{P}^2$. The Schur classes follow from the group structure \cite{EL2} and the Schubert relations.
\begin{align*}
c_1\left(V^{[3]}\right) &= (c_1-2r)H+rF  \\
c_2\left(V^{[3]}\right) &= 
\frac{1}{2}r (3 r - 1)A
+\left(c_1 r - \frac{1}{2}(3r-1)r\right)B
+\binom{c_1-2r+1}{2}C\\
&+\left(c_1 (2 r - 1) - r (3 r - 2)\right)D
+\left(c_2 - 2\binom{r}{2}\right)E\\
c_3\left(V^{[3]}\right) &= 
(c_1-2r+1)\left(c_2-(r-1)(c_1+2r)\right) U
+r\left(c_2-\frac{2}{3}(r-1)(2r-1)\right)V\\
&+\binom{c_1-2r+2}{3}W
+((7/3)(2r-1)r(r-1) +c_1 ((2r-1)c_1 - 6 r^2 + 6 r - 1))X\\
&+(c_1 \binom{3 r - 1}{2}- 6 \binom{r}{2}(2 r - 1))Y
+(r-1)(c_2+\frac{1}{6}(3 c_1 (c_1 - 3 r + 3) - r (r + 4)))Z\\
\end{align*}
\begin{align*}
c_4 \left(V^{[3]}\right) &= 
\frac{1}{4}(r-1) (-24 c_1 r^2 + 6 c_1^2 r + 36 c_1 r - 4 c_1^2 - 12 c_1 + 15 r^3 - 33 r^2 + 16 r)\alpha\\
&+(\frac{1}{2}(3 r - 2) (r - 1)c_2)\alpha\\
&+-\frac{1}{8}(r-1)(-36 c_1 r^2 + 24 c_1^2 r + 56 c_1 r - 4 c_1^3 - 20 c_1^2 - 16 c_1 + 13 r^3 - 29 r^2 + 14 r)\beta\\
&+((r-1)c_1-\frac{1}{2}(3r-2)(r-1))c_2\beta\\
&+(c_2-c_1 (r - 1))\binom{-c_1 + 2 r - 1}{2}\gamma\\
&+\frac{1}{2}(r - 1) (c_1 (2(2r-1)(r-1)-(3r-2)c_1)+ (r-1)^2r)\delta\\
&+((2r-1)c_1-(3r-1)(r-1))c_2\delta\\
&+\frac{1}{8}(r - 1) (3 r^3 - 3 r^2 - 2 r -2 c_1 (r - 2) (c_1 - 2 r + 3))\epsilon\\
&+\frac{1}{2}(c_2^2-2r^2+4r-3)c_2\epsilon\\
\end{align*}
\begin{align*}
c_5\left(V^{[3]}\right) &= (6\binom{r}{3}(r-1)(2r-3)-\frac{1}{4} c_1 (r - 1) (c_1^2 (r - 2) -(5 r - 6) (2 r - 3)c_1 + 20 r^3 - 72 r^2 + 82 r - 28))\phi\\
&+(\frac{1}{2} c_2 (-c_1 + 2 r - 2) (c_1(r-1) + 2 r^2 - 5 r + 4))\phi
+(\frac{1}{2}c_1-(r-1))c_2^2\phi\\
&-\frac{1}{40}(r-1)(40(r - 1) (2 r - 3)c_1^2 +(-255r^3 + 895r^2 - 1010r  + 360)c_1 + r (r - 2) (131 r^2 - 317 r + 192))\psi\\
&-\frac{1}{2}c_2 (r - 1) (- c_1^2 +3(r-1)c_1+ 3 r^2 - 8 r + 8)\psi
+(r-1)c_2^2)\psi
\end{align*}

\begin{align*}
c_6(V^{[3]}) &= (1/48)r^6\mu^6-(1/16)r^5\mu^6-(1/4)r^6\mu^4+(1/16)r^4\mu^6+(1/4)r^6\mu^3+(9/8)r^5\mu^4-(1/48)r^3\mu^6\\
&+(11/16)r^6\mu^2-(11/8)r^5\mu^3-(15/8)r^4\mu^4-(51/40)r^6\mu-(65/16)r^5\mu^2+(11/4)r^4\mu^3+(11/8)r^3\mu^4\\
&+(131/240)r^6+(71/8)r^5\mu+(461/48)r^4\mu^2-(19/8)r^3\mu^3-(3/8)r^2\mu^4-(71/16)r^5-(195/8)r^4\mu\\
&-(183/16)r^3\mu^2+(3/4)r^2\mu^3+(679/48)r^4+(265/8)r^3\mu+(55/8)r^2\mu^2-(1067/48)r^3-(447/20)r^2\mu\\
&-(5/3)r\mu^2+(2077/120)r^2+6r\mu-(16/3)r\\
&+((1/8)r^5\mu^4-(1/4)r^4\mu^4-(3/4)r^5\mu^2+(1/8)r^3\mu^4+(1/2)r^5\mu+(11/4)r^4\mu^2+(3/8)r^5-(9/4)r^4\mu)\Delta\\
&+(-(7/2)r^3\mu^2-(25/12)r^4+(13/4)r^3\mu+(3/2)r^2\mu^2+(41/8)r^3-(3/2)r^2\mu-(77/12)r^2+(10/3)r)\Delta\\
&+((1/4)r^4\mu^2-(1/4)r^3\mu^2-(1/2)r^4+(3/2)r^3-(3/2)r^2)\Delta^2
+(1/6)r^3\Delta^3
\end{align*}


\bibliographystyle{plain}
\bibliography{citations}

\end{document}